\documentclass[11pt]{article}
\usepackage{geometry}
\usepackage{color}
\usepackage{indentfirst}
\usepackage{algorithm}
\usepackage{algorithmic}
\usepackage{amsthm}
\usepackage{graphicx}
\usepackage{subfigure}
\usepackage{amsmath}
\usepackage{amssymb}
\usepackage{epstopdf}

\usepackage [latin1]{inputenc} 
\usepackage[T1]{fontenc} 
\usepackage{multirow}
\usepackage{fancyhdr}
\usepackage[title]{appendix}
\usepackage{enumerate}
\usepackage{tikz}
\usetikzlibrary{trees}
\usepackage{pifont}
\newtheorem{Lemma}{\textbf{Lemma}}
\newtheorem{Theorem}{\textbf{Theorem}}
\newtheorem{Example}{\textbf{Example}}

\newtheorem{Assumption}{\textbf{Assumption}}
\newtheorem{Definition}{\textbf{Definition}}
\newtheorem{remark}{Remark}

\date{\today}

\title{A quantum-inspired algorithm for approximating statistical leverage scores
\thanks{This work was supported by the National Key Research and Development Program of China No.\ 2021YFA1000600, the National Natural Science Foundation of China under Grant No.\ 11571265.}}

\author{Qian Zuo\footnotemark[2]\
\and Hua Xiang\footnotemark[2] \footnotemark[3] \footnotemark[4]}

\begin{document}
\maketitle
\renewcommand{\thefootnote}{\fnsymbol{footnote}}
\footnotetext[2]{School of Mathematics and Statistics, Wuhan
University, Wuhan 430072, China.} \footnotetext[3]{Hubei Key
Laboratory of Computational Science, Wuhan University, Wuhan 430072,
China.} \footnotetext[4]{E-mail address: {\tt hxiang@whu.edu.cn}.}

\vspace{-1em}
\begin{abstract}
Suppose a matrix $A \in \mathbb{R}^{m \times n}$ of rank $r$ with singular value decomposition $A = U_{A}\Sigma_{A} V_{A}^{T}$, where $U_{A} \in \mathbb{R}^{m \times r}$, $V_{A} \in \mathbb{R}^{n \times r}$ are orthonormal and $\Sigma_{A} \in \mathbb{R}^{r \times r}$ is a diagonal matrix. The statistical leverage scores of a matrix $A$ are the squared row-norms defined by $\ell_{i} = \|(U_{A})_{i,:}\|_2^2$, where $i \in [m]$, and the matrix coherence is the largest statistical leverage score. These quantities play an important role in machine learning algorithms such as matrix completion and Nystr\"{o}m-based low rank matrix approximation as well as large-scale statistical data analysis applications, whose usual algorithm complexity is polynomial in the dimension of the matrix $A$. As an alternative to the conventional approach, and inspired by recent development on dequantization techniques, we propose a quantum-inspired algorithm for approximating the statistical leverage scores. We then analyze the accuracy of the algorithm and perform numerical experiments to illustrate the feasibility of our algorithm. Theoretical analysis shows that our novel algorithm takes time polynomial in an integer $k$, condition number $\kappa$ and logarithm of the matrix size.
{\small
\begin{description}
\item[{\bf Keywords:}] Statistical leverage scores, Matrix coherence, Sample model, Randomized algorithm.\\
\end{description}
}
\end{abstract}

\section{Introduction}\label{sec:QI:SLSMC:Indro}\noindent
Conceptually, statistical leverage scores give a measure of the relative importance of each row. In \cite{HW78}, Hoaglin and Welsch first used statistical leverage scores as a regression diagnosis of least squares to detect potential outliers in the data. They describe how the size of the $i$-th statistical leverage score gives a measure of the influence of the $i$-th row on the least squares data fitting. If the $i$-th statistical leverage score is large, then one may doubt the accuracy of the corresponding data point, and may want to remove the abnormal data points in order to better fit the remaining data. Recently, Drineas et al. \cite{DMM061,DMM062,MWM11} advocate the use of statistical leverage scores in many randomized matrix algorithms. These random sampling algorithms for solving matrix problems are important tools to design fast algorithms, like least-squares regression \cite{DMMS11} and low rank matrix approximation \cite{TS06,DMM08,MD09,BMD09}. A related concept is matrix coherence, defined by the largest statistical leverage score, and it has gained widespread attention on recent popular problems, see \cite{CR09,TR10} and the references therein.

We can apply SVD or QR decomposition to obtain an orthogonal basis of the input matrix, then we compute the squared Euclidean norm of the rows of the orthogonal matrix to gain the statistical leverage scores. The matrix coherence can be obtained by finding the largest statistical leverage scores. The total time complexity is $O\left({\rm min}(mn^2,m^2n)\right)$, which is extremely slow if $m$ and $n$ are large. Many fast algorithms have been developed for their approximations \cite{DW12,LMP13,MMI10,LZ17,CP17,NN13,CMM17,SG21}, to name just a few. Drineas et al. \cite{DW12} present a randomized Hadamard transform algorithm based on subsampling to approximate statistical leverage scores in $O((mn + n^3){\rm log}\,m)$ time, which is much faster than the traditional algorithm for the case when $m \gg n \gg 1$. Recently, Sobczyk  et al. \cite{SG21} propose a approach based on rank revealing methods with compositions of dense and sparse randomized dimensionality reduction transforms. The running time complexity for dense and sparse matrices are about $O(nnz(A) + n^4 + nnz(A_{:,\mathcal{K}})\frac{{\rm ln}\, n}{\epsilon^2})$ and $O(nnz(A) + n^4 + nnz(A_{:,\mathcal{K}}))$, respectively, where $nnz(A)$ is the number of non-zero entries of $A$, $A_{:,\mathcal{K}}$ is the submatrix of $A$ containing the columns defined by $\mathcal{K}$, $\mathcal{K} \subseteq [n]$, $|\mathcal{K}| = k$, $k \leq {\rm rank}(A)$ and $\epsilon \in (0,1)$.

The first quantum algorithm to solve an $n \times n$ sparse linear system was proposed by Harrow, Hassidim, and Lloyd  in 2009 (referred to as the HHL algorithm) \cite{HHL09}. It outputs a quantum state $|x\rangle$ proportional to the answer $x$, and the running time complexity  is $O\left(({\rm log}\,n)s^2\kappa^2/\epsilon\right)$, where $\kappa$ is the condition number, $s$ is the most non-zero entries in each row of matrix $A$ and $\epsilon$ is the precision parameter. Compared with any classical version method, the HHL algorithm achieves an exponential speed up by computing the quantum state of the solution. Inspired by the HHL algorithm, a series of logarithmic time quantum algorithms have been proposed to solve various machine learning related problems, such as least square approximation \cite{WBL12,SXL20,HX19,KP20,SSP16}, linear system equations \cite{CGJ19,CJS13,18WZP,18WZP,SH18}, principal component analysis \cite{LMR14}, support vector machine \cite{RML14}, recommendation systems \cite{KP16}, etc \cite{CKD12,SX20,ZSWX21}. Like the HHL algorithm, the above logarithmic time algorithms do not provide a complete description of the solution, but a quantum state for sampling. Liu and Zhang \cite{LZ17} give a fast quantum algorithm based on amplitude amplification and amplitude estimation to approximate the statistical leverage scores of $A$ in $O({\rm log}\,(m + n)s^2 \kappa / \epsilon)$ time, which is an exponential speedup over the best known classical algorithms.
Note that the classical algorithm might also complete sampling from the solution in logarithmic time.
Recently, Tang \cite{ET18} proposes a classical algorithm for recommendation systems within logarithmic time by using the efficient low rank approximation techniques of Frieze, Kannan and Vempala \cite{FKV04}.
Such dequantizing techniques are applied to other low-rank matrix problems, such as matrix inversion \cite{GST18,CLW18,CGLLTW20}, singular value transformation \cite{JGS19}, non-negative matrix factorization \cite{CLS19}, support vector machine \cite{DBH21}, general minimum conical hull problems \cite{DHLT20}. Dequantizing techniques in those algorithms involve two technologies, the Monte-Carlo singular value decomposition and sampling techniques, which can be efficiently performed on low rank matrices.

Motivated by the work in Frieze et al. \cite{FKV04} and Tang \cite{ET18}, in this paper we present a quantum-inspired classical algorithm to compute the statistical leverage scores under some natural sampling assumptions. More specifically, based on the subsampling method in \cite{FKV04}, we obtain a small submatrix $W \in \mathbb{R}^{p \times p}$, its singular values and the corresponding right singular vectors. Then, we use the submatrix to exploit a new concise representation of the approximate left singular matrix $\hat{U}$ of $A$, i.e., $\hat{U} = S V\Sigma^{-1} \in \mathbb{R}^{m \times k}$, where $S \in \mathbb{R}^{m \times p}$, $V \in \mathbb{R}^{p \times k}$, and $\Sigma \in \mathbb{R}^{k \times k}$ are our subsampled low rank approximations to $A$, $V_{A}$, and $\Sigma_{A}$, respectively (from $A$'s singular value decomposition in Definition \ref{QI:SLSMC:Def:L1}). Next, we use the approximate left singular matrix $\hat{U}$ of $A$ to calculate the statistical leverage scores $\ell_{i}$, where $i \in [m]$. The core ideas are as follows. Firstly, based on $\hat{U}_{i,:} = S_{i,:} V\Sigma^{-1}$, we present a logarithmic time to estimate $S_{i,:} V$, which can be achieved by Lemma \ref{QI:SLSMC:Lemma:xy} for all $j \in [k]$ to approximate the inner product $S_{i,:} V_{:,j}$, resulting a 1-by-$k$ vector $\mathrm{t}^{T}$. Secondly, we calculate the matrix-vector multiplication $\tilde{U}_{i,:} = \mathrm{t}^{T} \Sigma^{-1}$, and then compute the inner product $\tilde{\ell}_{i} = \tilde{U}_{i,:}\tilde{U}_{i,:}^{T}$. Finally, we show that $\tilde{\ell}_{i}$ is a good approximation to $\ell_{i}$ in Theorem \ref{QI:SLSeMC:thm:M2}. Our algorithm takes time polynomial in integer $k$, condition number $\kappa$ and logarithm of the matrix size. The proposed algorithm achieves an exponential speedup in matrix size over any other classical algorithms for approximating statistical leverage scores.

Let us start with the following definition of the statistical leverage scores of a matrix.
\begin{Definition}\cite{LZ17}\label{QI:SLSMC:Def:L1}
Suppose a matrix $A \in \mathbb{R}^{m \times n}$ of rank $r$ with singular value decomposition $A = U_{A}\Sigma_{A} V_{A}^{T}$ where $U_{A} \in \mathbb{R}^{m \times r}$, $V_{A} \in \mathbb{R}^{n \times r}$ are orthonormal and $\Sigma_{A} \in \mathbb{R}^{r \times r}$ is a diagonal matrix. The statistical leverage scores of $A$ are defined as $\ell_{i} = \|(U_{A})_{i,:}\|_2^2, i \in [m]$, where $(U_{A})_{i,:}$ is the $i$-th row of $U_{A}$. The matrix coherence of $A$ is defined as $c = {\rm max}_{i \in [m]}\ell_{i}$, which is the largest statistical leverage score of $A$.
\end{Definition}

In this work, we use $\|A\|$, $\|A\|_F$ to denote the spectral norm and the Frobenius norm of the matrix $A$, respectively. For any $1 \leq i \leq n$, denote the $i$-th row of $A$ as $A_{i,:}$, and the $j$-th column of $A$ as $A_{:,j}$. Let $e_j$ be the $j$-th column of unit matrix. The $(i,j)$-entry of $A$ is denoted by $A_{i,j}$. Similarly, $v_{i}$ stands for the $i$-th entry of a vector $v$. We use $A^T$ to denote the transpose of $A$, and let $I \in \mathbb{R}^{m \times m}$ be identity matrix. The usual condition number is defined by $\kappa = \|A\| \|A^{-1}\|$, which is not well defined for a singular matrix. Here we slightly change the definition of the condition number of a singular matrix $A$, defined by $\kappa = \|A\|/\sigma_{\rm {min}}(A)$, where $\sigma_{\rm {min}}(A)$ is the minimum nonzero singular value of $A$. For a nonzero vector $x \in{\mathbb{R}^n}$, we denote by $\mathcal{D}_x$ the probability distribution on $i \in [n]$ whose probability density function is defined as $\mathcal{D}_x(i) = \frac{|x_{i}|^2}{\|x\|^2}$, where $[n]:= \{1,2,\cdots,n\}$. A sample from $\mathcal{D}_x$ is often referred to as a sample from $x$. In the problem of computing statistical leverage scores $\ell_{i}$, we need to query and sample from $(U_{A})_{i,:}$ for a given matrix $U_{A}$ and an index $i \in [m]$. Our main result is summarized as follows.

\begin{Theorem}\label{QI:SLSMC:thm:M2}
Suppose that a matrix $A \in \mathbb{R}^{m \times n}$ satisfies the sample model and data structure. Let $\epsilon \in (0,1)$ be an error parameter, $k \in \mathbb{N_{+}}$, $\delta \in (0,1)$ the failure probability, $\kappa$ the condition number, and recall that the definition of the statistical leverage scores $\ell$ from Definition \ref{QI:SLSMC:Def:L1}. Then, for $i \in [m]$, there exists an algorithm (Algorithm \ref{QI:Algorithm:SLSMC:FASLS:a2} of Section \ref{sec:QI:SLSMC:Algorithm} below) can return the approximate statistical leverage scores $\tilde{\ell}_{i}$, s.t. $|\ell_{i} - \tilde{\ell}_{i}| < \epsilon$, with probability $1-\delta$ by using $O\left({\rm poly} \left(k, \kappa, \frac{1}{\epsilon}, \frac{1}{\delta}, \frac{\|A\|_F}{\|A\|}, {\rm log}(mn)\right) \right)$ queries and time.
\end{Theorem}

The remainder of this work is organized as follows. In Section \ref{sec:QI:SLSMC:Sample:Model}, we review some basic concepts of the sample model and data structure, and then we describe the subsampling algorithm and our fast algorithm for statistical leverage scores  in Section \ref{sec:QI:SLSMC:Algorithm}. Next, in Section \ref{sec:QI:SLSMC:Algorithm:Analysis} we give a detailed algorithm analysis. Finally, we end this paper with some conclusions in Section \ref{sec:QI:SLSMC:conclusion}.

\section{Sample model and data structure}\label{sec:QI:SLSMC:Sample:Model}
We are interested in the development of sublinear time algorithms for computing statistical leverage scores, and we need to focus on how to input the given matrix and vector. Obviously, it is not possible to load the entire matrix and vector into memory because it costs at least linear time. In this work, we assume that matrices and vectors can be sampled according to some natural probability distributions, which are found in many applications of machine learning \cite{KP16,ET18,CLW18}.

We introduce the sampling assumption of the matrix in the following. Intuitively speaking, we assume that the column index can be sampled according to the norms of the column vectors, also an entry can be sampled according to the absolute values of the entry in that column.

\begin{Assumption}(\cite{CLW18})\label{QI:SLSMC:asu1}
Given a matrix $M \in \mathbb{R}^{m \times n}$, the following conditions hold.
\begin{enumerate}[1)]
  \item We can sample a column index $j \in [n]$ of $M$, where the probability of column $j$ being chosen is
  \begin{equation}\label{QI:SLSMC:P}
    P_j = \frac{\|M_{:,j}\|^2}{||M||_F^2}.
  \end{equation}
  \item For all $j \in [n]$, we can sample an index $i \in [m]$ according to $\mathcal{D}_{M_{:,j}}$, i.e.,  the probability of $i$ being chosen is
  \begin{equation}\label{QI:SLSMC:DM}
    \mathcal{D}_{M_{:,j}}(i) = \frac{|M_{i,j}|^2}{||M_{:,j}||^2}.
  \end{equation}
\end{enumerate}
\end{Assumption}

The above assumption actually is empirical. Frieze et al. \cite{FKV04} used these similar assumptions to present a sublinear algorithm for seeking low rank approximation. In \cite{ET18}, Tang also give a quantum-inspired classical algorithm for recommendation systems by using these similar assumptions. As pointed out in \cite{CLW18,CGLLTW20,CLS19,ET182}, there is a low-overhead data structure that satisfies the sampling assumption.
We first describe the data structure for a vector, then for a matrix.\\

\begin{Lemma}(Vector sample model)(\cite{CLW18})\label{QI:SLSMC:lemma:sample:V}
There exists a data structure storing a vector $v \in \mathbb{R}^{n}$ with $s$ nonzero entries in $O(s{\rm log}\,n)$ space, with the following properties:
\begin{enumerate}[a)]
  \item Querying and updating an entry of $v$ in $O({\rm log}\,n)$ time;
  \item Finding $||v||^2$ in $O(1)$ time;
  \item Sampling from $\mathcal{D}_{v}$ in $O({\rm log}\,n)$ time.
\end{enumerate}
\end{Lemma}
In \cite{ET18}, Tang gives a similar binary search tree (BST) diagram to analyze this data structure, as shown in the following figure.
\begin{figure}[H]
\centering
\begin{tikzpicture}[level distance=1.5cm,
  level 1/.style={sibling distance=3.5cm},
  level 2/.style={sibling distance=1.8cm}]
  \node {$||v||^2$}
    child {node {$v_1^2 + v_2^2$}
      child {node {$v_1^2$}
        child {node {${\rm sgn}(v_1)$}}
        }
      child {node {$v_2^2$}
        child {node {${\rm sgn}(v_2)$}}
        }
    }
    child {node {$v_3^2 + v_4^2$}
    child {node {$v_3^2$}
        child {node {${\rm sgn}(v_3)$}}
        }
      child {node {$v_4^2$}
        child {node {${\rm sgn}(v_4)$}}
        }
    };
\end{tikzpicture}
\caption{Binary search tree (BST) data structure for $v = (v_1, v_2, v_3, v_4)^{T} \in \mathbb{R}^{4}$.}
\end{figure}
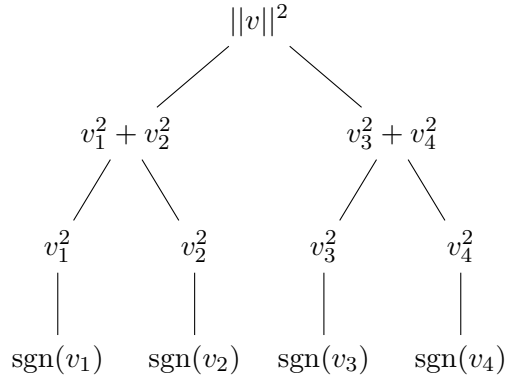

\begin{Lemma}(Matrix sample model)(\cite{CLW18})\label{QI:SLSMC:lemma:sample:M}
Given a matrix $A \in \mathbb{R}^{m \times n}$ with $s$ nonzero entries in $O(s{\rm log}\,mn)$ space,
with the following properties:
\begin{enumerate}[a)]
  \item Querying and updating an entry of $A$ in $O({\rm log}\,mn)$ time;
  \item Sampling from $A_{i,:}$ for any $i \in [m]$ in $O({\rm log}\,n)$ time;
  \item Sampling from $A_{:,j}$ for any $j \in [n]$ in $O({\rm log}\,m)$ time;
  \item Finding  $||A||_F$, $\|A_{i,:}\|$ and $\|A_{:,j}\|$ in $O(1)$ time.
\end{enumerate}
\end{Lemma}
Chia et al. \cite{CGLLTW20} present a similar binary search tree (BST) diagram for a matrix $A$ to analyze this data structure, as shown in the following figure.
\begin{figure}[H]
\begin{tikzpicture}[level distance=1.3cm,
  level 1/.style={sibling distance=8.8cm},
  level 2/.style={sibling distance=4.2cm},
  level 3/.style={sibling distance=2cm}]
  \node {$||A||_F^2$}
    child {node {$||A_{1,:}||^2$}
            child {node {$||A_{1,1}||^2+||A_{1,2}||^2$}
                child {node {$||A_{1,1}||^2$}
                    child {node {$\frac{A_{1,1}}{|A_{1,1|}}$}}
                }
                child {node {$||A_{1,2}||^2$}
                    child {node {$\frac{A_{1,2}}{|A_{1,2}|}$}}
                }
        }
            child {node {$||A_{1,3}||^2+||A_{1,4}||^2$}
                child {node {$||A_{1,3}||^2$}
                    child {node {$\frac{A_{1,3}}{|A_{1,3}|}$}}
                }
                child {node {$||A_{1,4}||^2$}
                    child {node {$\frac{A_{1,4}}{|A_{1,4}|}$}}
                }
        }
    }
    child {node {$||A_{2,:}||^2$}
        child {node {$||A_{2,1}||^2+||A_{2,2}||^2$}
                child {node {$||A_{2,1}||^2$}
                    child {node {$\frac{A_{2,1}}{|A_{2,1}|}$}}
                }
                child {node {$||A_{2,2}||^2$}
                    child {node {$\frac{A_{2,2}}{|A_{2,2}|}$}}
                }
        }
        child {node {$||A_{2,3}||^2+||A_{2,4}||^2$}
                child {node {$||A_{2,3}||^2$}
                    child {node {$\frac{A_{2,3}}{|A_{2,3}|}$}}
                }
                child {node {$||A_{2,4}||^2$}
                    child {node {$\frac{A_{2,4}}{|A_{2,4}|}$}}
                }
        }
    };
\end{tikzpicture}
\caption{Dynamic data structure for $A \in \mathbb{R}^{2 \times 4}$.}
\end{figure}
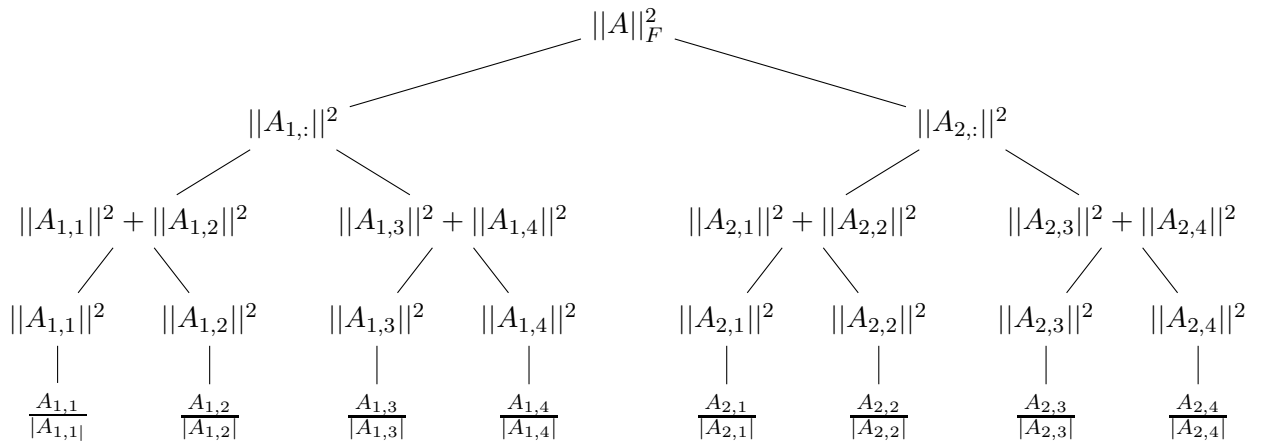

\section{Algorithm description}\label{sec:QI:SLSMC:Algorithm}
\subsection{Subsampling Algorithms}\label{sec:QI:SLSMC:Low:rank:Sample:Model}
 In this section, we show a subroutine (Algorithm \ref{QI:SLSMC:Subsampling:a1}) to generate the approximate left singular matrix, which approximates the left singular matrix $U_{A}$ of $A$. Firstly, let us recall the FKV algorithm \cite{FKV04}, which is a Monte-Carlo algorithm. It can find a good low rank approximation of a given matrix $A$. Without outputting the full matrix, it outputs a succinct description of the matrix $\widetilde{V}$ instead, where $\widetilde{V}$ is an approximation to the right singular matrix $V_{A}$ of matrix $A$, and the corresponding low rank approximation is $\widetilde{A}:= A \widetilde{V} \widetilde{V}^T$. Then, according to the definition of the statistical leverage scores of a matrix $A$, by a similar argument we can use the FKV algorithm to obtain the approximation left singular vectors $\widehat{U}$ of a given matrix $A$, and have the corresponding low rank approximation $\widehat{A} := \widehat{U} \widehat{U}^TA$.

We use FKV($A,k,\epsilon,\delta$) in \cite{FKV04} as follows, where the input matrix $A$ satisfies the sample model in Assumption \ref{QI:SLSMC:asu1}, $k$ is the input integer, $\epsilon$ is the error parameter, and $\delta$ is the failure probability.

\begin{Theorem}(\cite{FKV04})\label{QI:SLSMC:Thm:FKV1}
Given a matrix $A \in \mathbb{R}^{m \times n}$ in matrix sample model, $k \in \mathbb{N_{+}}$ and $\epsilon, \delta \in (0,1)$, FKV($A,k,\epsilon,\delta$) can output the approximate right singular matrix $\widetilde{V} \in \mathbb{R}^{n \times k}$ in $O({\rm poly}(k, \frac{1}{\epsilon}, {\rm log}\,\frac{1}{\delta}))$ samples and queries of $A$ with probability $1-\delta$, which satisfies
\begin{equation*}
    \|A \widetilde{V} \widetilde{V}^{T} - A\|_F^2 \leq \min \limits_{\widetilde{A}:\,{\rm rank}(\widetilde{A})\leq k} \|A - \widetilde{A} \|_F^2 + \epsilon \|A\|_F^2.
\end{equation*}
\end{Theorem}

Similarly, if we run FKV($A^{T},k,\epsilon,\delta$), we can obtain the following equivalent form of theorem.
\begin{Theorem}(\cite{CLS19})\label{QI:SLSMC:Thm:FKV2}
Given a matrix $A \in \mathbb{R}^{m \times n}$ in matrix sample model, $k \in \mathbb{N_{+}}$ and $\epsilon, \delta \in (0,1)$, FKV($A^{T},k,\epsilon,\delta$) can output the approximate left singular matrix $\widehat{U} \in \mathbb{R}^{m \times k}$ in $O({\rm poly}(k, \frac{1}{\epsilon}, {\rm log}\,\frac{1}{\delta}))$ samples and queries of $A$ with probability $1-\delta$, which satisfies
\begin{equation*}
    \|\widehat{U} \widehat{U}^TA - A\|_F^2 \leq \min \limits_{\widehat{A}:\,{\rm rank}(\widehat{A})\leq k} \|A - \widehat{A}\|_F^2 + \epsilon \|A\|_F^2.
\end{equation*}
\end{Theorem}
Next, we give a subroutine to introduce how to gain the approximate left singular matrix $\hat{U}$, and the resulting compact representation allows for efficiently sampling any row of $\hat{U}$ and querying any of its elements.
\begin{algorithm}[H]
\caption{Quantum-inspired singular value decomposition (QiSVD) algorithm.}\label{QI:SLSMC:Subsampling:a1}
\begin{algorithmic}[1]
\REQUIRE The matrix $A \in \mathbb{R}^{m \times n}$ that satisfies the sample model and data structure, the condition number $\kappa$, $k \in \mathbb{N_{+}}$, $\epsilon \in (0,1)$, $\delta \in (0,1)$, $\omega = \frac{\|A\|^2 \epsilon^2}{196(\|A\|_F \kappa + \|A\|)^2}$, ${\theta \in \left(0, \frac{\omega\|A\|^2}{(4k+3+2\omega)\kappa^2\|A\|^2_F}\right]}$ and $p = \left\lceil\frac{1}{\theta^2 \delta}\right\rceil$.\\
\ENSURE  The approximate left singular matrix $\hat{U}$.\\
\STATE  Independently sample $p$ column indices $j_1, j_2, \cdots, j_p$ according to the probability distribution $\{P_1, P_2, \cdots, P_n\}$ defined in Assumption \ref{QI:SLSMC:asu1}; \label{QI:SLSMC:Subsampling:al:s1}
\STATE  Let $S \in \mathbb{R}^{m \times p}$ be the matrix formed by the normalized columns $\frac{{A}_{:,j_t}}{\sqrt{pP_{j_t}}}$ for $t \in [p]$, i.e., $S_{:,t} = \frac{{A}_{:,j_t}}{\sqrt{pP_{j_t}}}$; \label{QI:SLSMC:Subsampling:al:s2}
\STATE Independently sample $p$ row indices $i_1, i_2, \cdots, i_p$ according to the probability distribution $\{P_1', P_2', \cdots, P_m'\}$, where $P_i' = \sum\limits_{t=1}^p \frac{\mathcal{D}_{A_{:,j_t}}(i)}{p}$, and $j_1, j_2, \cdots, j_p$ are the column indices in step \ref{QI:SLSMC:Subsampling:al:s1}; \label{QI:SLSMC:Subsampling:al:s3}
\STATE Let $W \in \mathbb{R}^{p \times p}$ be the matrix formed by the normalized rows $\frac{S_{i_t,:}}{\sqrt{pP_{i_t}'}}$ for $t \in [p]$, i.e., $W_{t,:} = \frac{S_{i_t,:}}{\sqrt{pP_{i_t}'}}$; \label{QI:SLSMC:Subsampling:al:s4}
\STATE Compute the largest $k$ singular values $\sigma_1, \sigma_2, \cdots, \sigma_k$ of matrix $W$ and their corresponding right singular vectors $v_1, v_2, \cdots, v_k$;\label{QI:SLSMC:Subsampling:al:s5}
\STATE Output the approximation left singular matrix $\hat{U} = S V \Sigma^{-1}$, where $V = (v_1, v_2, \cdots, v_k) \in \mathbb{R}^{p \times k}$, and $\Sigma = {\rm diag}(\sigma_1, \sigma_2, \cdots, \sigma_k) \in \mathbb{R}^{k \times k}$.
\end{algorithmic}
\end{algorithm}
\begin{remark}
The core ideas and sampling techniques in Algorithm \ref{QI:SLSMC:Subsampling:a1} are the same as in prior work \cite{ET18,FKV04,GST18,JGS19,ET182}, but with the different parameters setting for choosing $p$ rows and columns. Such parameters can make the error between the matrix $A$ and the low-rank matrix approximation $\hat{U}\hat{U}^{T}A$ smaller.
\end{remark}
\begin{remark}
In the step \ref{QI:SLSMC:Subsampling:al:s3} of Algorithm \ref{QI:SLSMC:Subsampling:a1}, combing the definition $\mathcal{D}_{A_{:,j}}(i)$ in Assumption \ref{QI:SLSMC:asu1} with Eq.\eqref{QI:SLSMC:Eq:SFAF}, it follows that
\begin{equation}\label{QI:SLSMC:Eq:PI}
    P_i' = \sum\limits_{t=1}^p \frac{\mathcal{D}_{A_{:,j_t}}(i)}{p} = \frac{1}{\|A\|_F^2}\sum\limits_{t=1}^p \frac{|A_{i,j_t}|^2}{p\frac{\|A_{:,j_t}\|^2}{\|A\|_F^2}} = \frac{1}{\|A\|_F^2}\sum\limits_{t=1}^p \frac{|A_{i,j_t}|^2}{pP_{j_t}} = \frac{\|S_{i,:}\|^2}{\|S\|_F^2}.
\end{equation}
\end{remark}
\begin{remark}
Suppose that $U_{A}$ is the exact left singular matrix of $A$ defined in Definition \ref{QI:SLSMC:Def:L1}, and Algorithm \ref{QI:SLSMC:Subsampling:a1} returns its approximate left singular matrix $\hat{U}$, then with probability $1-\delta$ we can show that $\left\|\hat{U}\hat{U}^{T} - U_{A}U_{A}^{T}\right\| < \frac{1}{2}\epsilon$. The proof will be given in Theorem \ref{QI:SLSMC:Theorem:UI}.
\end{remark}
From Algorithm \ref{QI:SLSMC:Subsampling:a1}, the approximate left singular matrix $\hat{U}$ can be obtained in the following way. Let $S$ be the submatrix given by scaling the columns to $\frac{{A}_{:,j_t}}{\sqrt{pP_{j_t}}}$. This makes sense if we regard $S \in \mathbb{R}^{m \times p}$, $V \in \mathbb{R}^{p \times k}$, and $\Sigma \in \mathbb{R}^{k \times k}$ as our subsampled low rank approximations of $A$, $V_{A}$, and $\Sigma_{A}$ respectively (from $A$'s singular value decomposition in Definition \ref{QI:SLSMC:Def:L1}). The details are given as follows.
\begin{itemize}
  \item $S$, a submatrix whose columns are sampled from columns of $A$ and normalized to equal norm $\frac{||A||_F}{\sqrt{p}}$;
  \item $V$, an explicit $p$-by-$k$ orthonormal matrix;
  \item $\Sigma$, an explicit diagonal matrix whose diagonal entries are $\sigma_i$, where $i\in [k]$.
\end{itemize}

In the above Algorithm \ref{QI:SLSMC:Subsampling:a1}, the core ideas are $SS^{T} \approx AA^{T}$ and $W^{T}W \approx S^{T}S$. The following lemma gives a proof for $W^{T}W \approx S^{T}S$ in \cite{FKV04}, then we generalize the proof from $W^{T}W \approx S^{T}S$  to $SS^{T} \approx AA^{T}$.

\begin{Lemma}(\cite{FKV04})\label{QI:SLSMC:Lemma:MTMNTN}
Given a matrix $M \in \mathbb{R}^{m \times n}$, let $P' = \{P'_{1}, P'_{2}, \cdots, P'_{m}\}$ be a probability distribution on $[m]$ such that $P'_{i} = \frac{\|M_{i,:}\|^2}{\|M\|^2_F}, i \in [m]$. Let $(i_1, i_2, \cdots, i_p)$ be a sequence of $p$ independent samples from $[m]$, each chosen according to distribution $P'$. Let $N$ be the ${p \times n}$ matrix with
\begin{equation}\label{QI:SLSMC:Eq:NN}
    N_{t,:} = \frac{M_{i_t,:}} {\sqrt{pP'_{i_t}}}, \, t \in [p].
\end{equation}
Then, For all $\theta > 0$, it holds that
\begin{equation}\label{QI:SLSMC:Eq:Pmmnn}
    {\rm Pr} (\|M^{T}M-N^{T}N\|_F \geq \theta \|M\|_F^2 ) \leq \frac{1}{\theta^2 p}.
\end{equation}
\end{Lemma}

\begin{Lemma}\label{QI:SLSMC:Lemma:MMTNNT}
Given a matrix $M \in \mathbb{R}^{m \times n}$, let $P = \{P_{1}, P_{2}, \cdots, P_{n}\}$ be a probability distribution distribution on $[n]$ such that $P_{j} = \frac{\|M_{:,j}\|^2}{\|M\|^2_F}, j \in [n]$. Let $(j_1, j_2, \cdots, j_p)$ be a sequence of $p$ independent samples from $[n]$, each chosen according to distribution $P$. Let $N$ be the ${m \times p}$ matrix with
\begin{equation}\label{QI:SLSMC:Eq:NN}
    N_{:,t} = \frac{M_{:,j_t}}  {\sqrt{pP_{j_t}}}, \, t \in [p].
\end{equation}
Then, for all $\theta > 0$, it yields that
\begin{equation}\label{QI:SLSMC:Eq:Pmmnn}
    {\rm Pr} (\|MM^{T} - NN^{T}\|_F \geq \theta \|M\|_F^2 ) \leq \frac{1}{\theta^2 p}.
\end{equation}
\end{Lemma}

\begin{proof}
Taking expectation for $NN^{T}$, it yields that
\begin{equation}\label{QI:SLSMC:Eq:mmnn}
    \begin{aligned}
       &~~~~\mathbb{E} (N_{a,:}N_{b,:}^{T}) = \sum_{t=1}^{p}\mathbb{E} (N_{a,j_t}N_{b,j_t})  \\
       & = \sum_{t=1}^{p}\sum_{j=1}^{n} P_{j} \frac{M_{a,j}} {\sqrt{pP_{j}}} \frac{M_{b,j}} {\sqrt{pP_{j}}}
       = \sum_{t=1}^{p}\sum_{j=1}^{n}  \frac{M_{a,j} M_{b,j} }{p}
       = \sum_{j=1}^{n} M_{a,j} M_{b,j}
       = M_{a,:}M_{b,:}^{T}.
    \end{aligned}
\end{equation}

Since $N_{a,:}N_{b,:}^{T} = \sum_{t=1}^{p} N_{a,j_t}N_{b,j_t}$ is the sum of $p$ independent random variables, using Eq.\eqref{QI:SLSMC:Eq:NN}, it follows that
\begin{equation}\label{QI:SLSMC:Eq:Vmmnn}
    \begin{aligned}
    &~~~~ {\rm Var}(N_{a,:}N_{b,:}^{T})
    = {\rm Var} \left(\sum_{t=1}^{p}N_{a,j_t}N_{b,j_t}\right) \\
        & = \sum_{t=1}^{p} {\rm Var} (N_{a,j_t}N_{b,j_t}) = \sum_{t=1}^{p} [ \mathbb{E}(N_{a,j_t} N_{b,j_t})^2 - (\mathbb{E}(N_{a,j_t}N_{b,j_t}))^2 ]\\
        & \leq \sum_{t=1}^{p} \mathbb{E}(N_{a,j_t}^2 N_{b,j_t}^2)  = \sum_{t=1}^{p}\sum_{j=1}^{n} P_{j} \frac{M_{a,j}^2M_{b,j}^2}{p^2 {P_{j}}^2}  \\
        & = \frac{\|M\|_F^2}{ p^2}\sum_{t=1}^{p}\sum_{j=1}^{n} \frac{M_{a,j}^2M_{b,j}^2}{\|M_{:,j}\|^2}  = \frac{\|M\|_F^2}{p} \sum_{j=1}^{n} \frac{M_{a,j}^2M_{b,j}^2}{\|M_{:,j}\|^2}.
    \end{aligned}
\end{equation}
Finally, taking expectation for the whole $\|MM^{T} - NN^{T}\|_F^2$, we have
\begin{equation}
    \begin{aligned}
        \mathbb{E}(\|MM^{T} - NN^{T}\|_F^2)
        & = \sum_{a,b=1}^{m} \mathbb{E} [ (N_{a,:}N_{b,:}^{T} - M_{a,:}M_{b,:}^{T})^2 ] \\
        & = \sum_{a,b=1}^{m} \mathbb{E} [ ( N_{a,:}N_{b,:}^{T} - \mathbb{E}(N_{a,:}N_{b,:}^{T}) )^2 ] \ \ \ \ ({\rm using}\, {\rm the}\, {\rm result}\, {\rm of}\, {\rm Eq}.\eqref{QI:SLSMC:Eq:mmnn})  \\
        & = \sum_{a,b=1}^{m} {\rm Var}(N_{a,:}N_{b,:}^{T}) \\
        & \leq \frac{\|M\|_F^2}{p} \sum_{j=1}^{n} \frac{1}{\|M_{:,j}\|^2} \sum_{a,b=1}^{m}M_{a,j}^2M_{b,j}^2 \ \ \  ({\rm using}\, {\rm the}\, {\rm result}\, {\rm of}\, {\rm Eq}.\eqref{QI:SLSMC:Eq:Vmmnn})  \\
        & = \frac{\|M\|_F^2}{p} \sum_{j=1}^{n} \frac{1}{\|M_{:,j}\|^2} \left(\sum_{a=1}^{m}M_{a,j}^2\right)^2 = \frac{\|M\|_F^4}{p}.
    \end{aligned}
\end{equation}
Using the fact that for any random variable $x$, ${\rm Var}(x) = \mathbb{E}(x^2) - [\mathbb{E}(x)]^2 \leq \mathbb{E}(x^2)$, we obtain $${\rm Var}(\|MM^{T} - NN^{T}\|_F) \leq \frac{\|M\|_F^4}{p}.$$

Combining with the result of Chebyshev's inequality, it yields that
\begin{equation*}
    {\rm Pr} (\|MM^{T} - NN^{T}\|_F \geq \theta \|M\|_F^2) \leq  \frac{1}{(\theta \sqrt{p})^2} = \frac{1}{\theta^2 p}.
\end{equation*}
\hfill
\end{proof}
Lemma \ref{QI:SLSMC:Lemma:MTMNTN} and \ref{QI:SLSMC:Lemma:MMTNNT} imply the following relations $W^{T}W \approx S^{T}S$ and $SS^{T} \approx AA^{T}$, respectively. More precisely, using the Lemma \ref{QI:SLSMC:Lemma:MTMNTN} and $p = \lceil\frac{1}{\theta^2 \delta}\rceil$ in Algorithm \ref{QI:SLSMC:Subsampling:a1}, it holds that
$${\rm Pr} (\|S^{T}S - W^{T}W\|_F \leq \theta \|S\|_F^2) \geq 1 - \delta.$$
That is, with probability at least $1 - \delta$, we have
\begin{equation}\label{QI:SLSMC:Eq:SSWW}
    \|S^{T}S - W^{T}W\|_F \leq \theta \|S\|_F^2.
\end{equation}
Similarly, according to Lemma \ref{QI:SLSMC:Lemma:MMTNNT}, it follows that
$${\rm Pr} (\|AA^{T} - SS^{T}\|_F \leq \theta \|A\|_F^2) \geq 1 - \delta.$$
Also, with probability at least $1- \delta$, we obtain
\begin{equation}\label{QI:SLSMC:Eq:AASS}
    \|AA^{T} - SS^{T}\|_F \leq \theta \|A\|_F^2.
\end{equation}

\subsection{Quantum-inspired Statistical Leverage Scores Algorithm}\label{sec:QI:SLSMC:Algorithm:FASLS}
By using the data structure in Lemma \ref{QI:SLSMC:lemma:sample:V}, the inner product of two vectors can be estimated quickly.

\begin{Lemma}(\cite{ET18})\label{QI:SLSMC:Lemma:xy}
Given $x, y \in \mathbb{R}^{n}$ in vector sample model, and query access to $x$ and $y$, the ability to sample from $\mathcal{D}_x$, one can approximate $\langle x, y \rangle$ to additive error $\xi \|x\|\|y\|$ with at least $1 - \eta$ probability using $O(\frac{1}{\xi^2}\, {\rm log}\frac{1}{\eta})$ queries and samples.
\end{Lemma}

We present a randomized algorithm for computing statistical leverage scores which is called quantum-inspired statistical leverage scores (QiSLS) algorithm. Especially, the input matrix $A$ of QiSLS algorithm is given by matrix sample model, which is realized via a data structure described in Section \ref{sec:QI:SLSMC:Sample:Model}. The statistical leverage scores can be returned by QiSLS algorithm in time polynomial logarithmic of the matrix size.
\begin{algorithm}[!htb]
\caption{Quantum-inspired statistical leverage scores (QiSLS) algorithm.}\label{QI:Algorithm:SLSMC:FASLS:a2}
\begin{algorithmic}[1]
\REQUIRE The matrix $A \in \mathbb{R}^{m \times n}$ satisfying the sample model and data structure, the condition number $\kappa$, the parameters $(\epsilon, \delta, k)$ in the specified range of Algorithm \ref{QI:SLSMC:Subsampling:a1}, $\xi = \frac{1}{\sqrt{k}}\left(\sqrt{\frac{2\epsilon \|A\|^2}{\kappa^2(4k+5)\|A\|^2_F}+1}-1\right)$ and $i \in [m]$.
\ENSURE The statistical leverage scores $\tilde{\ell}_{i}$.\\
\STATE Run Algorithm \ref{QI:SLSMC:Subsampling:a1} to gain the corresponding approximate left singular matrix $\hat{U} = S V \Sigma^{-1}$, i.e., $\hat{U}_{i,:} = S_{i,:} V \Sigma^{-1}$, $i \in [m]$.\label{QI:Algorithm:SLSMC:FASLS:a2:s1}
\STATE For all $j \in [k]$, $S_{i,:} V_{:,j}$ can be estimated by Lemma \ref{QI:SLSMC:Lemma:xy}, each with a precision $\xi \|S\|_F$ and success probability $(1-\delta)^{\frac{1}{k}}$. Let $\mathrm{t}^{T}$ be the resulting $1$-by-$k$ vector.\label{QI:Algorithm:SLSMC:FASLS:a2:s2}
\STATE Calculate $\tilde{U}_{i,:} = \mathrm{t}^{T} \Sigma^{-1}$ with element-wise division.\label{QI:Algorithm:SLSMC:FASLS:a2:s3}
\STATE Compute the inner product $\tilde{\ell}_i = \tilde{U}_{i,:}\tilde{U}^{T}_{i,:}$. \label{QI:Algorithm:SLSMC:FASLS:a2:s4}
\end{algorithmic}
\end{algorithm}

\section{Algorithm Analysis}\label{sec:QI:SLSMC:Algorithm:Analysis}
\begin{Lemma}(Weyl's inequality \cite{HW12})\label{QI:SLSMC:Lemma:SmSn}
For two matrices $M \in \mathbb{C}^{m \times n}$, $N \in \mathbb{C}^{m \times n}$ and any $i\in [{\rm min}(m,n)]$, $|\sigma_{i}(M) - \sigma_{i}(N)| \leq \|M - N\|$.
\end{Lemma}

\begin{Lemma}\label{QI:SLSMC:Lemma:mmw}
Given a matrix $A \in \mathbb{R}^{m \times n}$ with ${\rm rank}(A) = r$ satisfying the sample model and data structure, the condition number $\kappa$, and parameters $(p, \omega, \delta, k)$ in the specified range of Algorithm \ref{QI:SLSMC:Subsampling:a1}, Algorithm \ref{QI:SLSMC:Subsampling:a1} outputs the matrix $W \in \mathbb{R}^{p \times p}$, and with probability at least $1 - \delta$, it holds that
\begin{equation}
    \sigma_{\rm min}(W) \geq \sqrt{\frac{4k+3}{4k+3+2\omega}}\frac{\|A\|}{\kappa}.
\end{equation}
\end{Lemma}
\begin{proof}
Motivated by the work in \cite{CLW18,JGS19}, and by Lemma \ref{QI:SLSMC:Lemma:SmSn} we have
\begin{equation}\label{QI:SLSMC:Eq:sigmaAA}
  |\sigma_{\tilde{r}}(AA^{T}) - \sigma_{\tilde{r}}(SS^{T})|  \leq \|AA^{T} - SS^{T}\|,
\end{equation}
where $\tilde{r} = {\rm rank}(S)$. Combining Eq.\eqref{QI:SLSMC:Eq:AASS} with \eqref{QI:SLSMC:Eq:sigmaAA}, according to ${\theta \in \left(0, \frac{\omega\|A\|^2}{(4k+3+2\omega)\kappa^2\|A\|^2_F}\right]}$ in Algorithm \ref{QI:SLSMC:Subsampling:a1}, with probability at least $1-\delta$, it holds that

\begin{equation}\label{QI:SLSMC:Eq:sigmaSSTAAT}
    |\sigma_{\tilde{r}}(SS^{T}) - \sigma_{\tilde{r}}(AA^{T})| \leq \|AA^{T} - SS^{T}\| \leq \|AA^{T} - SS^{T}\|_F \leq \theta \|A\|_F^2 \leq  \frac{\omega\|A\|^2}{(4k+3+2\omega)\kappa^2}.
\end{equation}
Due to the definition of $\sigma_{\rm min}(A) = \frac{\|A\|}{\kappa}$, we have $\sigma_{\rm min}(AA^{T}) = \frac{\|A\|^2}{\kappa^2}$. Since $S$ is formed by sampling from the columns of matrix $A$, we obtain ${\rm rank}(SS^{T}) \leq {\rm rank}(AA^{T})$, i.e., $\tilde{r} \leq r$.
According to Eq.\eqref{QI:SLSMC:Eq:sigmaSSTAAT}, it follows that
\begin{equation}\label{QI:SLSMC:Eq:sigmaminS}
\sigma_{\tilde{r}}(SS^{T}) \geq \sigma_{\tilde{r}}(AA^{T}) - \frac{\omega\|A\|^2}{(4k+3+2\epsilon)\kappa^2} \geq \sigma_{\rm min}(AA^{T}) - \frac{\omega\|A\|^2}{(4k+3+2\omega)\kappa^2} = \frac{(4k+3+\omega)\|A\|^2}{(4k+3+2\omega)\kappa^2}.
\end{equation}

For any column $t \in [p]$ of matrix $S$, using the step \ref{QI:SLSMC:Subsampling:al:s2} in Algorithm \ref{QI:SLSMC:Subsampling:a1}, it yields that
\begin{equation}\label{QI:SLSMC:Eq:SFAF}
    \|S\|_F^2 = \sum_{t=1}^{p}\|S_{:,t}\|^2 = \sum_{t=1}^{p}\left\|\frac{A_{:,j_t}}{\sqrt{pP_{j_t}}}\right\|^2 = \sum_{t=1}^{p}\frac{\|A_{:,j_t}\|^2}{p\frac{\|A_{:,j_t}\|^2}{||A||_F^2}} = \sum_{t=1}^{p}\frac{\|A\|_F^2}{p} = \|A\|_F^2.
\end{equation}

By a similar analysis, combining Eq.\eqref{QI:SLSMC:Eq:SSWW} with \eqref{QI:SLSMC:Eq:SFAF}, with probability at least $1-\delta$, and by Lemma \ref{QI:SLSMC:Lemma:SmSn} we obtain
\begin{equation}\label{QI:SLSMC:Eq:sigmaSTSWTW}
    \begin{aligned}
     |\sigma_{\hat{r}}(S^{T}S) - \sigma_{\hat{r}}(W^{T}W)|
     \leq  \|S^{T}S - W^{T}W\|
     \leq  \|S^{T}S - W^{T}W\|_F
     \leq  \theta \|S\|_F^2
     \leq \frac{\omega\|A\|^2}{(4k+3+2\omega)\kappa^2},
    \end{aligned}
\end{equation}
where $\hat{r} = {\rm rank}(W)$, ${\theta \in \left(0, \frac{\omega\|A\|^2}{(4k+3+2\omega)\kappa^2\|A\|^2_F}\right]}$. Since $W$ is formed by sampling from the rows of matrix $S$, we obtain ${\rm rank}(W^{T}W) \leq {\rm rank}(S^{T}S)$, i.e., $\hat{r} \leq \tilde{r}$. Based on $\sigma_{\tilde{r}}(SS^{T}) \geq \frac{(4k+3+\omega)\|A\|^2}{(4k+3+2\omega)\kappa^2}$ in Eq.\eqref{QI:SLSMC:Eq:sigmaminS} and $\sigma_{\tilde{r}}(S^{T}S) = \sigma_{\tilde{r}}(SS^{T})$, the inequality \eqref{QI:SLSMC:Eq:sigmaSTSWTW} gives
\begin{equation}\label{QI:SLSMC:Eq:sigmaWWTW}
\sigma_{\hat{r}}(W^{T}W) \geq \sigma_{\hat{r}}(S^{T}S) - \frac{\omega\|A\|^2}{(4k+3+2\omega)\kappa^2} \geq \sigma_{\tilde{r}}(S^{T}S) - \frac{\omega\|A\|^2}{(4k+3+2\omega)\kappa^2} \geq \frac{(4k+3)\|A\|^2}{(4k+3+2\omega)\kappa^2}.
\end{equation}
Since $\sigma_{\rm min}(W^{T}W) = \sigma_{\hat{r}}(W^{T}W)$, we obtain
\begin{equation}\label{QI:SLSMC:Eq:sigmaW}
    \sigma_{\rm min}(W) \geq \sqrt{\frac{4k+3}{4k+3+2\omega}}\frac{\|A\|}{\kappa}.
\end{equation}
\hfill
\end{proof}

To summarize, in the problem of computing statistical leverage scores $\ell_{i}$, we need to calculate  $(U_{A})_{i,:}$ for a given matrix $U_{A}$ and an index $i \in [m]$. In Algorithm \ref{QI:SLSMC:Subsampling:a1} we employ right singular vectors of $W$ to approximate right singular vectors of $S$, and use the left singular vectors of $S$ to approximate left singular vectors of $A$. Then, we achieve the approximate left singular matrix $\hat{U} \in \mathbb{R}^{m \times k}$, which will suffice to make a good approximation to the exact left singular matrix $U_{A} \in \mathbb{R}^{m \times r}$ of $A$ defined in Definition \ref{QI:SLSMC:Def:L1}. Note that the following Theorem \ref{QI:SLSMC:Theorem:UI} will give a detailed proof of the approximation relation. We start with several lemmas to facilitate the proof of Theorem \ref{QI:SLSMC:Theorem:UI}.

\begin{Lemma}\label{QI:SLSMC:Lemma:B}
Given a matrix $A \in \mathbb{R}^{m \times n}$ satisfies the sample model and data structure, the parameters $(\omega, \delta, k)$ in the specified range of Algorithm \ref{QI:SLSMC:Subsampling:a1}, Algorithm \ref{QI:SLSMC:Subsampling:a1} outputs the approximate left singular matrix $\hat{U}$, then with probability $1- \delta$, it holds that
\begin{equation}\label{QI:SLSMC:Eq:U1}
    \begin{aligned}
        \|\hat{U}^{T}\hat{U}\| & \leq 1 + \beta, \\
        \|\hat{U}\| & \leq \sqrt{1 + \beta}, \\
                \|\hat{U}\|_F & \leq \sqrt{k\left(1 + \beta \right)},
    \end{aligned}
\end{equation}
where $\beta = \frac{\omega}{4k+3}$.
\end{Lemma}

\begin{proof}
According to $\hat{U} = S V \Sigma^{-1}$, based on the Eqs.\eqref{QI:SLSMC:Eq:SSWW}, \eqref{QI:SLSMC:Eq:SFAF} and \eqref{QI:SLSMC:Eq:sigmaW}, with probability at least $1 - \delta$, we have
\begin{equation}\label{QI:SLSMC:Eq:UU1}
  \begin{aligned}
    &~~~~\|\hat{U}^{T}\hat{U} - I\| \\
    & = \|\Sigma^{-T} V^{T}S^{T} SV\Sigma^{-1} - \Sigma^{-T} V^{T}W^{T}WV\Sigma^{-1}  \| \\
    & = \|\Sigma^{-T} V^{T}(S^{T}S-W^{T}W)V\Sigma^{-1}\| \leq \|\Sigma^{-1}\|^2 \| V\|^2 \|S^{T}S-W^{T}W\| \\
    & \leq \frac{\theta \|S\|_F^2}{\sigma^2_{\rm min}(W)} \leq \frac{\omega\|A\|^2}{(4k+3+2\omega)\kappa^2\|A\|^2_F} \|A\|_F^2 \frac{ (4k+3+2\omega)\kappa^2}{(4k+3)\|A\|^2} = \beta.
  \end{aligned}
\end{equation}

Based on $ \left\|\hat{U}^{T}\hat{U} - I\right\| \leq \beta$, we have
\begin{equation}\label{QI:SLSMC:Eq:UUE3}
\|\hat{U}^{T}\hat{U}\| = \|I + \hat{U}^{T}\hat{U} - I\| \leq \| I \| + \| \hat{U}^{T}\hat{U} - I \| \leq 1 + \beta.
\end{equation}

By Lemma \ref{QI:SLSMC:Lemma:SmSn} and Eq.\eqref{QI:SLSMC:Eq:UU1}, for any $i \in [k]$, it holds that
\begin{equation}\label{QI:SLSMC:Eq:U13}
    |\sigma_{i}(\hat{U}^{T}\hat{U}) - \sigma_{i}(I^{T}I)| \leq  \|\hat{U}^{T}\hat{U} - I\| \leq \beta,
\end{equation}
which implies the following relationship $\sigma_{\rm max}(\hat{U}^{T}\hat{U}) \leq 1+ \beta$. Moreover, we have
\begin{equation}\label{QI:SLSMC:Eq:BU2}
    \|\hat{U}\| \leq \sqrt{1 + \beta}.
\end{equation}

Since $\hat{U} = SV\Sigma^{-1}$, i.e., $\hat{U}_{:,j}= \frac{SV_{:,j}}{\sigma_j}\in \mathbb{R}^{m}$, for $j \in [k]$, with probability at least $1 - \delta$, it yields that
  \begin{equation}\label{QI:SLSMC:Eq:UUE}
  \begin{aligned}
    | \|\hat{U}_{:,j}\|^2 - 1|
     = \left|\frac{(V_{:,j})^{T}(S^{T}S - W^{T}W)V_{:,j}}{\sigma_j^2}\right|
     \leq \frac{\|S^{T}S - W^{T}W\|}{\sigma_j^2}
     \leq \frac{\theta \|S\|_F^2}{\sigma^2_{\rm min}(W)}
     \leq \beta.
  \end{aligned}
  \end{equation}
That is, we obtain
\begin{equation}\label{QI:SLSMC:Eq:UTF}
    \|\hat{U}\|_F = \sqrt{\sum_{j=1}^{k} \|\hat{U}_{:,j}\|^2} \leq \sqrt{k (1 + \beta )}.
\end{equation}
\hfill
\end{proof}

\begin{remark}
Similar to Eq.\eqref{QI:SLSMC:Eq:UU1}, with probability at least $1 - \delta$, it yields that
\begin{equation}\label{QI:SLSMC:Eq:UUF3}
    \|\hat{U}^{T}\hat{U} - I\|_F \leq \|\Sigma^{-1}\|^2 \| V\|^2 \|S^{T}S-W^{T}W\|_F \leq \beta.
\end{equation}
\end{remark}

\begin{Lemma}\label{QI:SLSMC:Lemma:B}
Suppose that a matrix $A \in \mathbb{R}^{m \times n}$ satisfies the sample model and data structure, the parameters $(\omega, \delta, k)$ in the specified range of Algorithm \ref{QI:SLSMC:Subsampling:a1}. Algorithm \ref{QI:SLSMC:Subsampling:a1} outputs the approximate left singular matrix $\hat{U}$, then with probability $1- \delta$, it follows that
\begin{equation}\label{QI:SLSMC:Eq:U12}
     \|A - \hat{U}\hat{U}^{T}A\|_F^2 < \|A - A_k\|_F^2 + 5\omega \|A\|_F^2,
\end{equation}
where $A_k = \sum_{i= 1}^{k} \sigma_{i}(A) (U_{A})_{:,i} [(V_{A})_{:,i}]^{T}$.
\end{Lemma}

\begin{proof}
Motivated by the work in \cite{PDKM06}, using the thin singular value decomposition of $\hat{U}\in \mathbb{R}^{m \times k}$, i.e.,
\begin{equation}\label{QI:SLSMC:Eq:USVD}
     \hat{U} = X \Sigma_{\hat{U}} Y^{T},
\end{equation}
where $X \in \mathbb{R}^{m \times k}$, $Y \in \mathbb{R}^{k \times k}$ are orthonormal and $\Sigma_{\hat{U}} \in \mathbb{R}^{k \times k}$ is a diagonal matrix, we obtain
\begin{equation}\label{QI:SLSMC:Eq:UUE2}
\|\hat{U}^{T}A\|_F^2 = \|\Sigma_{\hat{U}} X^{T} A \|_F^2 \leq \|\Sigma_{\hat{U}}\|^2 \|X^{T} A \|_F^2 = \|\hat{U}^{T}\hat{U}\| \|X^{T} A \|_F^2.
\end{equation}
Since $\omega = \frac{\|A\|^2 \epsilon^2}{196(\|A\|_F \kappa + \|A\|)^2} $ in Algorithm \ref{QI:SLSMC:Subsampling:a1}, it yields that $\omega< 1$ and $\beta = \frac{\omega}{4k+3} < 1$. Combining Eq.\eqref{QI:SLSMC:Eq:UUE3} with \eqref{QI:SLSMC:Eq:UUE2}, it follows that
\begin{equation}\label{QI:SLSMC:Eq:UUE4}
        \|X^{T} A \|_F^2 \geq \frac{1}{\|\hat{U}^{T}\hat{U}\|} \|\hat{U}^{T}A\|_F^2 \geq \frac{1}{1 + \beta}  \|\hat{U}^{T}A\|_F^2 > ({1 - \beta}) \|\hat{U}^{T}A\|_F^2.
\end{equation}
Since $\|\hat{U}^{T}A\|_F^2 =  {\rm Tr} (\hat{U}^{T}A A^{T}\hat{U} ) =  {\rm Tr} (\hat{U}^{T}SS^{T}\hat{U}) + {\rm Tr} [\hat{U}^{T}(AA^{T} - SS^{T})\hat{U}]$, using Eqs. \eqref{QI:SLSMC:Eq:AASS} and \eqref{QI:SLSMC:Eq:UTF}, then with probability $1- \delta$, we have
\begin{equation}\label{QI:SLSMC:Eq:UUE6}
    \begin{aligned}
        & ~~~~ |{\rm Tr} [\hat{U}^{T}(AA^{T} - SS^{T})\hat{U}] |
        = \left|\sum_{i=1}^{k} (\hat{U}^{T})_{i,:}(AA^{T} - SS^{T})\hat{U}_{:,i}\right| \\
        & \leq \sum_{i=1}^{k} |(\hat{U}^{T})_{i,:}(AA^{T} - SS^{T})\hat{U}_{:,i}|
         \leq \|AA^{T} - SS^{T}\| \|\hat{U}\|_F^2
         \leq \theta\|A\|_F^2 k(1+\beta),
    \end{aligned}
\end{equation}
and
\begin{equation}\label{QI:SLSMC:Eq:UUE7}
    \begin{aligned}
         \|\hat{U}^{T}A\|_F^2 = {\rm Tr} (\hat{U}^{T}SS^{T}\hat{U}) + {\rm Tr} [\hat{U}^{T}(AA^{T} - SS^{T})\hat{U}] \geq \|\hat{U}^{T}S\|_F^2 - \theta\|A\|_F^2 k(1+\beta).
    \end{aligned}
\end{equation}

According to the triangle inequality, it yields that
\begin{equation}\label{QI:SLSMC:Eq:UUE77}
    \begin{aligned}
         \left|\| W^{T}W V\Sigma^{-1}\|_F - \|(S^{T} S - W^{T}W) V\Sigma^{-1}\|_F\right|
         & \leq \| W^{T}W V\Sigma^{-1}+(S^{T}S-W^{T}W)V\Sigma^{-1}\|_F \\
         & = \|S^{T}SV\Sigma^{-1}\|_F = \|S^{T} \hat{U}\|_F,
    \end{aligned}
\end{equation}
which implies the following relationship
\begin{equation}\label{QI:SLSMC:Eq:UUE8}
    \begin{aligned}
         \|\hat{U}^{T}S\|_F^2
         & \geq (\| W^{T}W V\Sigma^{-1}\|_F - \|(S^{T} S - W^{T}W) V\Sigma^{-1}\|_F)^2 \\
         & = \left[\left(\sum_{t=1}^{k}\sigma^2_{t}(W)\right)^{\frac{1}{2}} - \|(S^{T} S - W^{T}W) V\Sigma^{-1}\|_F\right]^2 \\
         & \geq \left[\left(\sum_{t=1}^{k}\sigma^2_{t}(W)\right)^{\frac{1}{2}} - \frac{\theta\|S\|_F^2}{\sigma_{\rm min}(W)}\right]^2 \\
         & \geq \sum_{t=1}^{k}\sigma^2_{t}(W)- \frac{2\theta\|A\|_F^2}{\sigma_{\rm min}(W)}\|W\|_F.
    \end{aligned}
\end{equation}

Due to Cauchy-Schwarz inequality and Hoffman-Wielandt Theorem, it holds that
\begin{equation}\label{QI:TLS:Eq:SCE}
    \begin{aligned}
         & ~~~~ \left|\sum_{t=1}^{k}[\sigma^2_{t}(S) - \sigma^2_{t}(A)]\right|
         \leq \sqrt{k} \left(\sum_{t=1}^{k}[\sigma^2_{t}(S) - \sigma^2_{t}(A)]^2\right)^{\frac{1}{2}} \\
         & = \sqrt{k} \left(\sum_{t=1}^{k}[\sigma_{t}(S^{T}S) - \sigma_{t}(A^{T}A)]^2\right)^{\frac{1}{2}}
         \leq \sqrt{k} \|S^{T}S - A^{T}A\|_F
         \leq \sqrt{k} \theta \|S\|^2_F.
    \end{aligned}
\end{equation}
By a similar argument, we have
\begin{equation}\label{QI:TLS:Eq:WSE}
    \begin{aligned}
         & ~~~~ \left|\sum_{t=1}^{k}[\sigma^2_{t}(W) - \sigma^2_{t}(S)]\right|
         \leq \sqrt{k} \left(\sum_{t=1}^{k}[\sigma^2_{t}(W) - \sigma^2_{t}(S)]^2\right)^{\frac{1}{2}} \\
         & = \sqrt{k} \left(\sum_{t=1}^{k}[\sigma_{t}(WW^{T}) - \sigma_{t}(SS^{T})]^2\right)^{\frac{1}{2}}
         \leq \sqrt{k} \|SS^{T} - WW^{T}\|_F
         \leq \sqrt{k} \theta \|S\|^2_F.
    \end{aligned}
\end{equation}
Combining Eq.\eqref{QI:TLS:Eq:SCE} with \eqref{QI:TLS:Eq:WSE}, it follows that
\begin{equation}\label{QI:TLS:Eq:WCE}
\begin{aligned}
        \left|\sum_{t=1}^{k}[\sigma^2_{t}(W) - \sigma^2_{t}(A)]\right|
        & = \left|\sum_{t=1}^{k}[\sigma^2_{t}(W) - \sigma^2_{t}(S)  + \sigma^2_{t}(S)  - \sigma^2_{t}(A)]\right| \\
        & \leq \left|\sum_{t=1}^{k}[\sigma^2_{t}(W) - \sigma^2_{t}(S)]\right|  + \left|\sum_{t=1}^{k}[\sigma^2_{t}(S)  - \sigma^2_{t}(A)]\right| \\
        & \leq \sqrt{k} \|SS^{T} - WW^{T}\|_F + \sqrt{k} \|S^{T}S - A^{T}A\|_F
        = 2\sqrt{k} \theta \|S\|^2_F,
\end{aligned}
\end{equation}
which implies the following relationship
\begin{equation}\label{QI:TLS:Eq:SWCE}
\sum_{t=1}^{k}\sigma^2_{t}(W)  \geq \sum_{t=1}^{k}\sigma^2_{t}(A) - 2\sqrt{k} \theta \|S\|^2_F.
\end{equation}

For any row $t \in [p]$ of matrix $W$, using the step \ref{QI:SLSMC:Subsampling:al:s4} in Algorithm \ref{QI:SLSMC:Subsampling:a1}, it yields that
\begin{equation}\label{QI:SLSMC:Eq:WFSF}
    \|W\|_F^2 = \sum_{t=1}^{p}\|W_{t,:}\|^2 = \sum_{i=1}^{p}\left\|\frac{S_{i_t,:}}{\sqrt{pP'_{i_t}}}\right\|^2 = \sum_{t=1}^{p}\frac{\|S_{i_t,:}\|^2}{p\frac{\|S_{i_t,:}\|^2}{\|S\|_F^2}} = \sum_{t=1}^{p}\frac{\|S\|_F^2}{p} = \|S\|_F^2.
\end{equation}

Combining  Eq.\eqref{QI:SLSMC:Eq:SFAF} with \eqref{QI:SLSMC:Eq:WFSF}, we have $\|W\|_F = \|S\|_F = \|A\|_F$. Using Eqs.\eqref{QI:SLSMC:Eq:UUE4}, \eqref{QI:SLSMC:Eq:UUE7} and \eqref{QI:TLS:Eq:SWCE}, it follows that
\begin{equation}\label{QI:SLSMC:Eq:UUE9}
        \|X^{T} A \|_F^2 > ({1 - \beta})  \left[\sum_{t=1}^{k}\sigma^2_{t}(A) - 2\sqrt{k} \theta\|A\|_F^2 - \frac{2\theta\|A\|_F^3}{\sigma_{\rm min}(W)}- k\theta(1+\beta))\|A\|_F^2  \right].
\end{equation}

According to $\beta = \frac{\omega}{4k+3} \geq \frac{\theta \|S\|_F^2}{\sigma^2_{\rm min}(W)}$ in Eq.\eqref{QI:SLSMC:Eq:UU1} and $\|S\|_F^2 = \|W\|_F^2$, we have $\frac{\beta}{\theta} \geq \frac{\|S\|_F^2}{\sigma^2_{\rm min}(W)} = \frac{\|W\|_F^2}{\sigma^2_{\rm min}(W)} > 1$. Moreover, it holds that $1 < \frac{\|W\|_F}{\sigma_{\rm min}(W)} = \frac{\|A\|_F}{\sigma_{\rm min}(W)} \leq \frac{\|A\|_F^2}{\sigma^2_{\rm min}(W)}$. By direct computation, we obtain

\begin{equation}\label{QI:SLSMC:Eq:UUE0}
  \begin{aligned}
    & \quad \ \left\|A - XX^{T}A\right\|_F^2 \\
    & = {\rm Tr} [(A - XX^{T}A)^{T}(A - XX^{T}A)]  = {\rm Tr} [(A^{T}A - A^{T}XX^{T}A)] = \|A\|_F^2 - \|X^{T}A\|_F^2 \\
    & \leq \|A\|_F^2 - ({1 - \beta})\left[\sum_{t=1}^{k}\sigma^2_{t}(A) - 2\sqrt{k} \theta\|A\|_F^2 - \frac{2\theta\|A\|_F^3}{\sigma_{\rm min}(W)}- k\theta(1+\beta))\|A\|_F^2  \right] \\
    & \leq \|A\|_F^2 - \sum_{t=1}^{k}\sigma^2_{t}(A) + \beta \sum_{t=1}^{k}\sigma^2_{t}(A) +  2\sqrt{k} \theta\|A\|_F^2 + \theta k(1+\beta)\|A\|_F^2 + \frac{2\theta\|A\|_F^3}{\sigma_{\rm min}(W)} \\
    & \leq \|A - A_k\|_F^2 + \left[\beta  + 2\sqrt{k} \theta +  \theta k(1+\beta) + \frac{2\theta\|A\|_F}{\sigma_{\rm min}(W)}\right] \|A\|_F^2 \\
    & < \|A - A_k\|_F^2 + \left[\beta + 2k\beta + k\beta + k\beta^2 + \frac{2\theta\|A\|_F^2}{\sigma^2_{\rm min}(W)}\right] \|A\|_F^2 \\
    & < \|A - A_k\|_F^2 + (\beta  + 3k\beta + k\beta + 2\beta)\|A\|_F^2 \\
    & = \|A - A_k\|_F^2 + (4k + 3) \beta\|A\|_F^2 = \|A - A_k\|_F^2 + \omega\|A\|_F^2.
  \end{aligned}
\end{equation}

Using Eq.\eqref{QI:SLSMC:Eq:USVD} and \eqref{QI:SLSMC:Eq:UUF3}, we have
\begin{equation}\label{QI:SLSMC:Eq:BBUUE}
    \|XX^{T} - \hat{U}\hat{U}^{T}\|_F
    = \|X(I - \Sigma_{\hat{U}}^{2}) X^{T}\|_F   \\
    = \|I - \Sigma_{\hat{U}}^{2}\|_F \\
    = \|Y(I - \Sigma_{\hat{U}}^{2})Y^{T}\|_F \\
    = \|I - \hat{U}^{T}\hat{U}\|_F \\
    \leq \beta.
\end{equation}

By the triangle inequality, it yields that
\begin{equation}\label{QI:SLSMC:Eq:V^{T}BV}
\begin{aligned}
   \|A - \hat{U}\hat{U}^{T}A\|_F^2
   & \leq (\|A - XX^{T}A\|_F + \|XX^{T}A - \hat{U}\hat{U}^{T}A\|_F)^2 \\
   & \leq (1 + \omega)\|A - XX^{T}A\|_F^2 + (1 + 1/\omega)\|XX^{T}A - \hat{U}\hat{U}^{T}A\|_F^2  \\
   & \leq (1 + \omega)(\|A - A_k\|_F^2 + \omega\|A\|_F^2) + (1 + 1/\omega)\|XX^{T} - \hat{U}\hat{U}^{T}\|_F^2 \|A\|_F^2 \\
   & = \|A - A_k\|_F^2 + \omega \|A - A_k\|_F^2 + (1 + \omega)  \omega \|A\|_F^2 + (1 + 1/\omega) \beta^2 \|A\|_F^2 \\
   & \leq \|A - A_k\|_F^2 + \omega \|A\|_F^2 + (\omega^2 + \omega + \beta^2 + \beta^2 / \omega) \|A\|_F^2 \\
   & < \|A - A_k\|_F^2 + (2\omega^2 + 3\omega) \|A\|_F^2 \\
   & < \|A - A_k\|_F^2 + 5\omega \|A\|_F^2.
\end{aligned}
\end{equation}
\end{proof}

\begin{Lemma}(\cite{SJG91})\label{QI:SLSMC:Lemma:ERI}
Let $A$ be an $n \times n$ positive definite matrix and $A = LL^{H}$ its Cholesky factorization. If $E$ is an $n \times n$ Hermitian matrix satisfying $\|A^{-1}\| \|E\|_F < \frac{1}{2}$, then there is a unique Cholesky factorization
\begin{equation}\label{QI:SLSMC:Eq:ERI}
\begin{aligned}
    A + E & = (L + G)(L + G)^{H}, \\
    \frac{\|G\|_F}{\|L\|_p} & \leq \sqrt{2} \frac{\kappa \|E\|_F / \|A\|_p}{1 + \sqrt{1 - 2\kappa \|E\|_F / \|A\|}}, \ \ p = 2, F.
\end{aligned}
\end{equation}
\end{Lemma}

\begin{Theorem}\label{QI:SLSMC:Theorem:UI}
Consider a matrix $A \in \mathbb{R}^{m \times n}$ with ${\rm rank}(A) = r$ satisfies the sample model and data structure, the parameters $(\omega, \delta, k)$ in the specified range of Algorithm \ref{QI:SLSMC:Subsampling:a1}. Suppose that $U_{A}$ is the exact left singular matrix of $A$ defined in Definition \ref{QI:SLSMC:Def:L1}, and Algorithm \ref{QI:SLSMC:Subsampling:a1} outputs the approximate left singular matrix $\hat{U}$, then with probability $1- \delta$, it holds that
\begin{equation}\label{QI:SLSMC:Eq:U10}
    \|\hat{U}\hat{U}^{T} - U_{A}U_{A}^{T}\| < \frac{1}{2}\epsilon.
\end{equation}
\end{Theorem}

\begin{proof}
Using the QR decomposition of $\hat{U}$, i.e., $\hat{U} = Q\left[\begin{array}{c}
R \\
\mathbf{0}
\end{array}\right] \in \mathbb{R}^{m \times k}$, where $Q \in \mathbb{R}^{m \times m}$ be an orthonormal matrix and $R \in \mathbb{R}^{k \times k}$ be an upper triangular matrix. Then, applying
Eq.\eqref{QI:SLSMC:Eq:UUF3}, we have

\begin{equation}\label{QI:SLSMC:Eq:UURR}
    \|\hat{U}^{T}\hat{U} - I\|_F = \|R^{T}R - I\|_F \leq \beta.
\end{equation}
We find that $R^{T}R$ can be viewed as an approximate Cholesky decomposition of $I$. Since $0 < \omega < 1$ and $\beta = \frac{\omega}{4k+3}$, we have $0 < \beta < \frac{1}{7}$. Then, by Lemma \ref{QI:SLSMC:Lemma:ERI}, it yields that
\begin{equation}\label{QI:SLSMC:Eq:RIF}
    \|R - I\|_F \leq \frac{\sqrt{2}\beta}{1 + \sqrt{1-2\beta}} = \frac{\sqrt{2}}{2}(1 - \sqrt{1-2\beta}).
\end{equation}
Since $0 < \beta < \frac{1}{7}$, we have $\frac{\sqrt{2}(1 - \sqrt{1-2\beta})}{2\beta} = \frac{\sqrt{2}}{1 + \sqrt{1-2\beta}} < \frac{\sqrt{2}}{1 + \sqrt{\frac{5}{7}}} < 1$. The Eq.\eqref{QI:SLSMC:Eq:RIF} becomes
\begin{equation}\label{QI:SLSMC:Eq:UZRI}
    \|R - I\|_F \leq \frac{\sqrt{2}}{2}(1 - \sqrt{1-2\beta}) < \beta.
\end{equation}
Setting $Z = Q\left[\begin{array}{c}
I \\
\mathbf{0}
\end{array}\right] \in \mathbb{R}^{m \times k}$, combining with Eq.\eqref{QI:SLSMC:Eq:UZRI}, it follows that
\begin{equation}\label{QI:SLSMC:Eq:UZ}
    \|\hat{U} - Z\|_F  = \|R - I\|_F  < \beta.
\end{equation}

The matrix $Z$ builds a bridge to evaluate the upper bound of $\|\hat{U}\hat{U}^{T} - U_{A}U_{A}^{T}\|$ in Eq.\eqref{QI:SLSMC:Eq:U10}. In fact, we need to estimate two items $\|U_{A}U_{A}^{T} - ZZ^{T}\|$ and $\|ZZ^{T} - \hat{U}\hat{U}^{T}\|$. The second is somewhat easier. For the first item, we need the smallest singular value of $Z^{T}U_{A}$, which depends on the estimate on $\|(I - ZZ^{T})U_{A}\|$.

Setting $\hat{U} = Z + E$, with $\|E\|_F < \beta$, applying Eq.\eqref{QI:SLSMC:Eq:U1}, we obtain
\begin{equation}\label{QI:SLSMC:Eq:AUUA}
\begin{aligned}
   \|A - ZZ^{T}A\|_F
   & = \|A - (\hat{U}-E)(\hat{U}-E)^{T}A\|_F \\
   & = \|A - \hat{U}\hat{U}^{T}A  + \hat{U}E^{T}A + E\hat{U}^{T}A - EE^{T}A\|_F \\
   & \leq \|A - \hat{U}\hat{U}^{T}A\|_F + \|\hat{U}E^{T}A\|_F + \|E\hat{U}^{T}A\|_F + \|EE^{T}A\|_F \\
   & \leq \|A - \hat{U}\hat{U}^{T}A\|_F + 2\|E\|_F \|\hat{U}\| \|A\|_F + \|E\|_F^2 \|A\|_F \\
   & < \|A - \hat{U}\hat{U}^{T}A\|_F + (2\beta \sqrt{1+\beta} + \beta^2) \|A\|_F.
\end{aligned}
\end{equation}
Squaring both sides of the Eq.\eqref{QI:SLSMC:Eq:AUUA}, using Eq.\eqref{QI:SLSMC:Eq:V^{T}BV}, it follows that
\begin{equation}\label{QI:SLSMC:Eq:AUUA2}
\begin{aligned}
   \|A - ZZ^{T}A\|_F^2
   & \leq \left[\|A - \hat{U}\hat{U}^{T}A\|_F + (2\beta \sqrt{1+\beta} + \beta^2) \|A\|_F\right]^2 \\
   & \leq \|A - \hat{U}\hat{U}^{T}A\|^2_F + \left[2(2+\beta)(2\beta \sqrt{1+\beta} + \beta^2) + (2\beta \sqrt{1+\beta} + \beta^2)^2\right] \|A\|^2_F \\
   & < \|A - A_k\|^2_F + \Delta \|A\|_F^2,
\end{aligned}
\end{equation}
where $\Delta = 5\omega + 2 (2+\beta)(2\beta \sqrt{1+\beta} + \beta^2) + (2\beta \sqrt{1+\beta} + \beta^2)^2$.
Equivalently, it holds that

\begin{equation}\label{QI:SLSMC:Eq:AUUAE}
\begin{aligned}
   \Delta \|A\|_F^2
   & > \|A - ZZ^{T}A\|_F^2 - \|A - A_k\|^2_F\\
   & = \|A\|_F^2 - \|Z^{T}A\|_F^2 - (\|A\|^2_F - \|A_k\|^2_F) = \|A_k\|^2_F - \|Z^{T}A\|_F^2 \\
   & = \sum_{i=1}^{k}\sigma^2_{i}(A) - \|Z^{T} U_{A} \Sigma_{A} V_{A}^{T}\|_F^2 = \sum_{i=1}^{k}\sigma^2_{i}(A) - \sum_{i=1}^{r}\sigma^2_{i}(A) \|Z^{T} (U_{A})_{:,i}\|^2 \\
   & = \sum_{i=1}^{k}\sigma^2_{i}(A) \left(1 - \|Z^{T} (U_{A})_{:,i} \|^2\right) - \sum_{i=k +1}^{r}\sigma^2_{i}(A)\|Z^{T} (U_{A})_{:,i} \|^2\\
   & \geq \sigma^2_{k}(A) \sum_{i=1}^{k} \left(1 - \|Z^{T} (U_{A})_{:,i} \|^2\right) - \sigma^2_{k}(A) \sum_{i=k +1}^{r}\|Z^{T} (U_{A})_{:,i} \|^2\\
   & = \sigma^2_{k}(A) (k - \|Z^{T}U_{A}\|_F^2),
\end{aligned}
\end{equation}
which implies the following relationship
\begin{equation}\label{QI:SLSMC:Eq:AUUAEE}
\begin{aligned}
   \|Z^{T}U_{A}\|_F^2 > k - \frac{\Delta \|A\|_F^2}{\sigma^2_k(A)} > k - \frac{\Delta \|A\|_F^2}{\sigma^2_{\rm min}(A)} = k - \frac{\Delta \|A\|_F^2 \kappa^2}{\|A\|^2}.
\end{aligned}
\end{equation}

Using the singular value decomposition of $Z^{T}U_{A} \in \mathbb{R}^{k \times r}$, i.e., $U_{\tau}\Sigma_{\tau}V^{T}_{\tau}$, where $U_{\tau} \in \mathbb{R}^{k \times k}$, $V_{\tau} \in \mathbb{R}^{r \times r}$ are orthonormal and $\Sigma_{\tau} = {\rm diag}(\tau_1, \tau_2, \cdots, \tau_k) \in \mathbb{R}^{k \times r}$ is a diagonal matrix. Since $Z$ and $U_{A}$ are orthonormal column matrices, we have $1 \geq \tau_1 \geq \tau_2 \geq \cdots \geq \tau_k > 0$. That is, the Eq.\eqref{QI:SLSMC:Eq:AUUAEE} implies the following relationship
\begin{equation}\label{QI:SLSMC:Eq:muEE}
\begin{aligned}
   \tau^2_k & > k -  \frac{\Delta \|A\|_F^2 \kappa^2}{\|A\|^2} - (\tau^2_1 + \tau^2_2 + \cdots + \tau^2_{k-1}) \\
           & = 1 -  \frac{\Delta \|A\|_F^2 \kappa^2}{\|A\|^2} + (1 - \tau^2_1) + (1 - \tau^2_2) + \cdots + (1 - \tau^2_{k-1}) \\
           & \geq 1 -  \frac{\Delta \|A\|_F^2 \kappa^2}{\|A\|^2}.
\end{aligned}
\end{equation}

Using the results of Eqs.\eqref{QI:SLSMC:Eq:muEE}, it holds that
\begin{equation}\label{QI:SLSMC:Eq:UAUAZZ}
\begin{aligned}
 & ~~~~ \|U_{A}U_{A}^{T} - ZZ^{T}\|^2 = \|(I - ZZ^{T})U_{A}\|^2 \\
 & = 1 - \tau^2_k <  1 - \left(1 -  \frac{\Delta \|A\|_F^2 \kappa^2}{\|A\|^2}\right) = \frac{\Delta \|A\|_F^2 \kappa^2}{\|A\|^2}.
\end{aligned}
\end{equation}

By triangle inequality, using Eq.\eqref{QI:SLSMC:Eq:U1} and \eqref{QI:SLSMC:Eq:UZ}, it yields that
\begin{equation}\label{QI:SLSMC:Eq:UUZZ}
\begin{aligned}
    & ~~~~ \|ZZ^{T} - \hat{U}\hat{U}^{T}\|
    \leq \|ZZ^{T} - Z\hat{U}^{T}\| + \|Z\hat{U}^{T} - \hat{U}\hat{U}^{T}\|\\
    & \leq \|Z\| \|Z^{T} - \hat{U}^{T}\| + \|Z - \hat{U}\| \|\hat{U}^{T}\|\\
    & = \|Z - \hat{U}\|( \|Z\| + \|\hat{U}^{T}\|)\\
    & \leq \|Z - \hat{U}\|_F( \|Z\| + \|\hat{U}^{T}\|) < \beta(1 + \sqrt{1 + \beta}).
\end{aligned}
\end{equation}

Since $\beta = \frac{\omega}{4k+3}$ and $\omega = \frac{\|A\|^2 \epsilon^2}{196(\|A\|_F \kappa + \|A\|)^2} < 1$, we have $\beta < \omega < \sqrt{\omega} <1$. It follows that
\begin{equation}\label{QI:SLSMC:Eq:Delta}
\begin{aligned}
   \Delta
   & = 5\omega + 2 (2+\beta)(2\beta \sqrt{1+\beta} + \beta^2) + (2\beta \sqrt{1+\beta} + \beta^2)^2 \\
   & = 5\omega + (8\beta + 4\beta^2 + 2\beta^3)\sqrt{1+\beta} + (8\beta^2 + 6\beta^3 + \beta^4)\\
   & < 5\omega + 14\beta\sqrt{1+\beta} + 15\beta\\
   & < 5\omega + 28\beta + 15\beta \\
   & < 49 \omega.
\end{aligned}
\end{equation}
By direct computation, applying Eqs.\eqref{QI:SLSMC:Eq:UAUAZZ}, \eqref{QI:SLSMC:Eq:UUZZ} and \eqref{QI:SLSMC:Eq:Delta}, we obtain
\begin{equation}\label{QI:SLSMC:Eq:UUUT0}
\begin{aligned}
   & ~~~~ \| U_{A}U_{A}^{T} - \hat{U}\hat{U}^{T} \|
     = \| U_{A}U_{A}^{T} - ZZ^{T} + ZZ^{T} - \hat{U}\hat{U}^{T} \| \\
   & \leq \| U_{A}U_{A}^{T} - ZZ^{T} \| + \|ZZ^{T} - \hat{U}\hat{U}^{T} \| \\
   & < \sqrt{\frac{\Delta \|A\|_F^2 \kappa^2}{\|A\|^2}} + \beta(1 + \sqrt{1 + \beta})
   < \frac{\|A\|_F \kappa}{\|A\|} \sqrt{\Delta} + \sqrt{\Delta}\\
   & <  \frac{7(\|A\|_F \kappa + \|A\|) }{\|A\|} \sqrt{\frac{\|A\|^2 \epsilon^2}{196(\|A\|_F \kappa + \|A\|)^2}}
   = \frac{1}{2}\epsilon.
\end{aligned}
\end{equation}

\end{proof}

\begin{Theorem}\label{QI:SLSeMC:thm:M2}
Suppose that a matrix $A \in \mathbb{R}^{m \times n}$ satisfies the sample model and data structure, let $\epsilon \in (0,1)$ be an error parameter, $k \in \mathbb{N_{+}}$, $\delta \in (0,1)$ the failure probability, $\kappa$ the condition number, and recall that the definition of the statistical leverage scores $\ell$ from Definition \ref{QI:SLSMC:Def:L1}. Then, for $i \in [m]$, Algorithm \ref{QI:Algorithm:SLSMC:FASLS:a2} can return the approximate statistical leverage scores $\tilde{\ell}_{i}$, s.t. $|\ell_{i} - \tilde{\ell}_{i}| < \epsilon$, with probability $1-\delta$ by using $O\left({\rm poly} \left(k, \kappa, \frac{1}{\epsilon}, \frac{1}{\delta}, \frac{\|A\|_F}{\|A\|}, {\rm log}(mn)\right) \right)$ queries and time.
\end{Theorem}

\begin{proof}
In step \ref{QI:Algorithm:SLSMC:FASLS:a2:s1} of Algorithm \ref{QI:Algorithm:SLSMC:FASLS:a2}, recalling that $\hat{U} = SV\Sigma^{-1}$, it holds that
\begin{equation}\label{QI:SLSMC:Eq:U2svd}
|\hat{U}_{i,:}\| = \|S_{i,:}V\Sigma^{-1}\| \leq \|S_{i,:}\| \|V\| \|\Sigma^{-1}\| < \frac{\|S\|_F}{\sigma_{\rm min}(W)}.
\end{equation}
In step \ref{QI:Algorithm:SLSMC:FASLS:a2:s2} of Algorithm \ref{QI:Algorithm:SLSMC:FASLS:a2}, recalling that $\mathrm{t}^{T}_{j} \approx S_{i,:}V_{:,j}$, for all $j\in[k]$, it yields that each estimate of inner product has an additive error $\xi\|S_{i,:}\|\|V_{:,j}\|$ by Lemma \ref{QI:SLSMC:Lemma:xy}, each with probability $(1-\delta)^{\frac{1}{k}}$, i.e.,
$$|\mathrm{t}^{T}_{j} - S_{i,:}V_{:,j}| \leq \xi \|S_{i,:}\| \|V_{:,j}\| < \xi \|S\|_F.$$
Moreover, with probability $1- \delta$, we obtain
\begin{equation}\label{QI:SLSMC:Eq:w2sv}
\|\mathrm{t}^{T} - S_{i,:}V\| = \sqrt{\sum_{j=1}^{k}|\mathrm{t}^{T}_{j} - S_{i,:}V_{:,j}|^2} < \sqrt{k} \xi \|S\|_F.
\end{equation}
In step \ref{QI:Algorithm:SLSMC:FASLS:a2:s3} of Algorithm \ref{QI:Algorithm:SLSMC:FASLS:a2}, recalling that $\tilde{U}_{i,:} = \mathrm{t}^{T} \Sigma^{-1}$,
based on Eq.\eqref{QI:SLSMC:Eq:w2sv}, it yields that
\begin{equation}\label{QI:SLSMC:Eq:U2u}
\begin{aligned}
& ~~~~ \|\tilde{U}_{i,:} - \hat{U}_{i,:}\|
= \|\mathrm{t}^{T} \Sigma^{-1} - S_{i,:}V \Sigma^{-1}\| \\
& = \|(\mathrm{t}^{T} - S_{i,:}V)\Sigma^{-1}\|
\leq \|\mathrm{t}^{T} - S_{i,:}V\| \|\Sigma^{-1}\|
< \frac{\sqrt{k} \xi \|S\|_F}{\sigma_{\rm min}(W)}.
\end{aligned}
\end{equation}
Moreover, using the fact that $|\|\tilde{U}_{i,:}\| - \|\hat{U}_{i,:}\|| \leq \|\tilde{U}_{i,:} - \hat{U}_{i,:}\| < \frac{\sqrt{k} \xi \|S\|_F}{\sigma_{\rm min}(W)}$ and Eq.\eqref{QI:SLSMC:Eq:U2svd}, we have
\begin{equation}\label{QI:SLSMC:Eq:TU}
\|\tilde{U}_{i,:}\| < \|\hat{U}_{i,:}\| + \frac{\sqrt{k} \xi \|S\|_F}{\sigma_{\rm min}(W)} < \frac{\|S\|_F}{\sigma_{\rm min}(W)} + \frac{\sqrt{k} \xi \|S\|_F}{\sigma_{\rm min}(W)} = \frac{\|S\|_F}{\sigma_{\rm min}(W)}(1 + \sqrt{k}\xi).
\end{equation}
In step \ref{QI:Algorithm:SLSMC:FASLS:a2:s4} of Algorithm \ref{QI:Algorithm:SLSMC:FASLS:a2}, recalling that $\tilde{\ell}_i = \tilde{U}_{i,:}\tilde{U}^{T}_{i,:}$, defining $\hat{\ell}_{i} = \hat{U}_{i,:}\hat{U}^{T}_{i,:}$, it follows that
\begin{equation}\label{QI:SLSMC:Eq:l22l0}
    \begin{aligned}
   |\tilde{\ell}_{i} - \hat{\ell}_{i}|
    & = |\tilde{U}_{i,:}\tilde{U}^{T}_{i,:} - \hat{U}_{i,:}\hat{U}^{T}_{i,:}| \\
    & = |\tilde{U}_{i,:}\tilde{U}^{T}_{i,:} - \tilde{U}_{i,:}\hat{U}^{T}_{i,:} + \tilde{U}_{i,:}\hat{U}^{T}_{i,:} - \hat{U}_{i,:}\hat{U}^{T}_{i,:}|\\
    & \leq |\tilde{U}_{i,:}\tilde{U}^{T}_{i,:} - \tilde{U}_{i,:}\hat{U}^{T}_{i,:}| + |\tilde{U}_{i,:}\hat{U}^{T}_{i,:} - \hat{U}_{i,:}\hat{U}^{T}_{i,:}| \\
    & = |\tilde{U}_{i,:}(\tilde{U}_{i,:} - \hat{U}_{i,:})^{T}| + |(\tilde{U}_{i,:} - \hat{U}_{i,:})\hat{U}^{T}_{i,:}| \\
    & \leq \|\tilde{U}_{i,:}\| \|(\tilde{U}_{i,:} - \hat{U}_{i,:})^{T}\| + \|\tilde{U}_{i,:} - \hat{U}_{i,:}\| \|\hat{U}^{T}_{i,:}\| \\
    & = (\|\tilde{U}_{i,:} - \hat{U}_{i,:}\|) (\|\tilde{U}_{i,:}\| + \|\hat{U}^{T}_{i,:}\|) \\
    & < \frac{\|A\|_F^2}{\sigma^2_{\rm min}(W)}(k\xi^2 + 2\sqrt{k}\xi) \ \ \ \ ({\rm using}\, {\rm the}\, {\rm results}\, {\rm of}\, {\rm Eqs}.\eqref{QI:SLSMC:Eq:SFAF},\, \eqref{QI:SLSMC:Eq:U2svd},\, \eqref{QI:SLSMC:Eq:U2u}\, {\rm and} \eqref{QI:SLSMC:Eq:TU})  \\
    & < \frac{(4k + 3 + 2\omega)\|A\|_F^2\kappa^2}{(4k + 3)\|A\|^2}\left[(\sqrt{k}\xi + 1)^2 -1\right] \ \ \ \ ({\rm using}\, {\rm Lemma}\, \ref{QI:SLSMC:Lemma:mmw}) \\
    & < \frac{(4k + 5)\|A\|_F^2\kappa^2}{4\|A\|^2}\left[(\sqrt{k}\xi + 1)^2 -1\right] \\
    & = \frac{(4k + 5)\|A\|_F^2\kappa^2}{4\|A\|^2}\frac{2\epsilon\|A\|^2}{\kappa^2(4k + 5)\|A\|_F^2} \\
    & = \frac{1}{2}\epsilon.
    \end{aligned}
\end{equation}

Then, using the results of Eq.\eqref{QI:SLSMC:Eq:UUUT0} and \eqref{QI:SLSMC:Eq:l22l0}, we achieve the total error
\begin{equation}\label{QI:SLSMC:Eq:ll2ll}
    \begin{aligned}
    & ~~~~ |\ell_{i} - \tilde{\ell}_{i} |
    = |\ell_{i} - \hat{\ell}_{i} + \hat{\ell}_{i} - \tilde{\ell}_{i} | \\
    & \leq |\ell_{i} - \hat{\ell}_{i}| + |\hat{\ell}_{i} - \tilde{\ell}_{i}| \\
    & = |(U_{A})_{i,:}[(U_{A})_{i,:}]^{T} - \hat{U}_{i,:}\hat{U}^{T}_{i,:}| + |\hat{\ell}_{i} - \tilde{\ell}_{i}| \\
    & = |e_{i}^T U_{A}U_{A}^{T}e_{i} - e_{i}^T\hat{U}\hat{U}^{T}e_{i}| + |\hat{\ell}_{i} - \tilde{\ell}_{i}| \\
    & = |e_{i}^T(U_{A}U_{A}^{T} - \hat{U}\hat{U}^{T})e_{i}| + |\hat{\ell}_{i} - \tilde{\ell}_{i}| \\
    & \leq \|e_{i}^T\| \|U_{A}U_{A}^{T} - \hat{U}\hat{U}^{T}\| \|e_{i}\| + |\hat{\ell}_{i} - \tilde{\ell}_{i}| \\
    & =  \|U_{A}U_{A}^{T} - \hat{U}\hat{U}^{T}\|  + |\hat{\ell}_{i} - \tilde{\ell}_{i}|
    < \frac{1}{2}\epsilon + \frac{1}{2}\epsilon = \epsilon.
    \end{aligned}
\end{equation}

Let us analyze the time complexity of Algorithm \ref{QI:Algorithm:SLSMC:FASLS:a2}. Due to $p = \left\lceil\frac{1}{\theta^2 \delta}\right\rceil = \left\lceil\frac{(4k + 3 + 2\omega)^2 \kappa^4 \|A\|_F^4}{\omega^2 \|A\|^4 \delta}\right\rceil$ and $\omega = \frac{\|A\|^2 \epsilon^2}{196(\|A\|_F \kappa + \|A\|)^2}$ in Algorithm \ref{QI:SLSMC:Subsampling:a1}, the time complexity of Algorithm \ref{QI:Algorithm:SLSMC:FASLS:a2} is mainly dominated by the SVD decomposition of step \ref{QI:SLSMC:Subsampling:al:s5} of Algorithm \ref{QI:SLSMC:Subsampling:a1}. The mainly time complexity is
$$O(p^3{\rm poly\,log}\,(mn)) = O\left(\frac{k^{6} \kappa^{24} \|A\|_F^{24}}{\epsilon^{12} \delta^3 \|A\|^{24}}{\rm poly\,log}\,(mn)\right).$$

Next, the query complexity can be calculated in the following. \\

In step \ref{QI:Algorithm:SLSMC:FASLS:a2:s2} of Algorithm \ref{QI:Algorithm:SLSMC:FASLS:a2}, defining $\eta = 1-(1-\delta)^{\frac{1}{k}}$ with $\xi = \frac{1}{\sqrt{k}}\left(\sqrt{\frac{2\epsilon\|A\|^2}{\kappa^2(4k+5)\|A\|^2_F}+1}-1\right)$, the vector $\mathrm{t}$ can be estimated by Lemma \ref{QI:SLSMC:Lemma:xy}, i.e., for all $j \in [k]$, $\mathrm{t}_{j} \approx S_{i,:}V_{:,j}$, the mainly query complexity is
\begin{equation}\label{QI:SLSMC:Eq:ll2ll}
    \begin{aligned}
O\left(k\frac{1}{\xi^2}\, {\rm log}\frac{1}{\eta}\right)
& = O\left(k\left(\frac{\sqrt{k}}{\sqrt{\frac{2\epsilon\|A\|^2}{\kappa^2(4k+5)\|A\|^2_F}+1}-1}\right)^2{\rm log}\frac{1}{1-(1-\delta)^{\frac{1}{k}}}\right)\\
& = O\left(k^2 \left(\frac{\sqrt{\frac{2\epsilon\|A\|^2}{\kappa^2(4k+5)\|A\|^2_F}+1}+1}{\frac{2\epsilon\|A\|^2}{\kappa^2(4k+5)\|A\|^2_F}}\right)^2{\rm log}\frac{1}{1-(1-\delta)^{\frac{1}{k}}}\right)\\
& = O\left( \frac{k^2 \kappa^4 (4k+5)^2\|A\|^4_F}{\epsilon^2\|A\|^4} \left(\frac{2\epsilon\|A\|^2}{\kappa^2(2k+5)\|A\|^2_F}+1\right){\rm log}\frac{1}{1-(1-\delta)^{\frac{1}{k}}}\right) \\
& = O\left(\frac{k^4 \kappa^4\|A\|^4_F}{\epsilon^2\|A\|^4} {\rm log}\frac{1}{\delta} \right).
    \end{aligned}
\end{equation}
In step \ref{QI:Algorithm:SLSMC:FASLS:a2:s3} and \ref{QI:Algorithm:SLSMC:FASLS:a2:s4} of Algorithm \ref{QI:Algorithm:SLSMC:FASLS:a2}, the query complexity of calculation are $O(k)$ and $O(k)$, respectively.

Hence, the total queries complexity of algorithm are
$$O\left(\frac{k^6 \kappa^4\|A\|^4_F}{\epsilon^2\|A\|^4} {\rm log}\,\frac{1}{\delta}\,{\rm poly\,log}\,(mn)\right).$$
\hfill
\end{proof}

\section{Numerical experiments}\label{sec:QI:SLSMC:conclusion}
In this section, we illustrate the feasibility of the proposed quantum-inspired statistical leverage scores (QiSLS) algorithm in practice by testing on artificial datasets. The feasibility and efficiency of some other quantum-inspired algorithms has been benchmarked, such as general minimum conical hull problems, support vector machine and canonical correlation analysis in \cite{DBH21,DHLT20,KMKM21}. Here the numerical tests are performed on a laptop with Intel Core i5 by MATLAB R2016(a) with a machine precision of $10^{-16}$. To show the robustness of the sampling in our QiSLS algorithm, we run 50 times repeated trials and report the mean instead of just picking one time.
\begin{Example}\label{QI:SLSMC:Exp1}
In this example, followed the work by in \cite{HIW14}, we generate datasets $A$ and $B$ of size $1000 \times 100$ and $4000 \times 100$, respectively, where $A$ and $B$ are defined by
\begin{equation}\label{QI:SLSMC:Eq:Exp1}
    \begin{aligned}
    A & = {\rm diag}(I_{250}, \, 10^2I_{250}, \, 10^3I_{250}, \, 10^4I_{250})\, {\rm randn}(1000,100), \\
    t & = {\rm randperm}(100,n), \, A(:,t(i)) = {\rm zeros}(1000,1), i\in [n], \\
    B & = {\rm diag}(I_{1000}, \, 10^2I_{1000}, \, 10^3I_{1000}, \, 10^4I_{1000})\, {\rm randn}(4000,100), \\
    r & = {\rm randperm}(100,n), \, B(:,r(i)) = {\rm zeros}(4000,1), i\in [n],
    \end{aligned}
\end{equation}
where $n < 100$ is a given number, the rank of matrix $A$ and $B$ are $100-n$. The previous quantum-inspired algorithms in \cite{DBH21,DHLT20,KMKM21,ADBL20} indicate that quantum-inspired algorithms can perform well in practice under low-rank. We modified the structure of matrix $A$ and $B$ in the second and forth lines of Eq.\eqref{QI:SLSMC:Eq:Exp1}.
\end{Example}

\begin{figure}[H]
\begin{minipage}[t]{0.5\linewidth}
\centering
\includegraphics[width=2.5in]{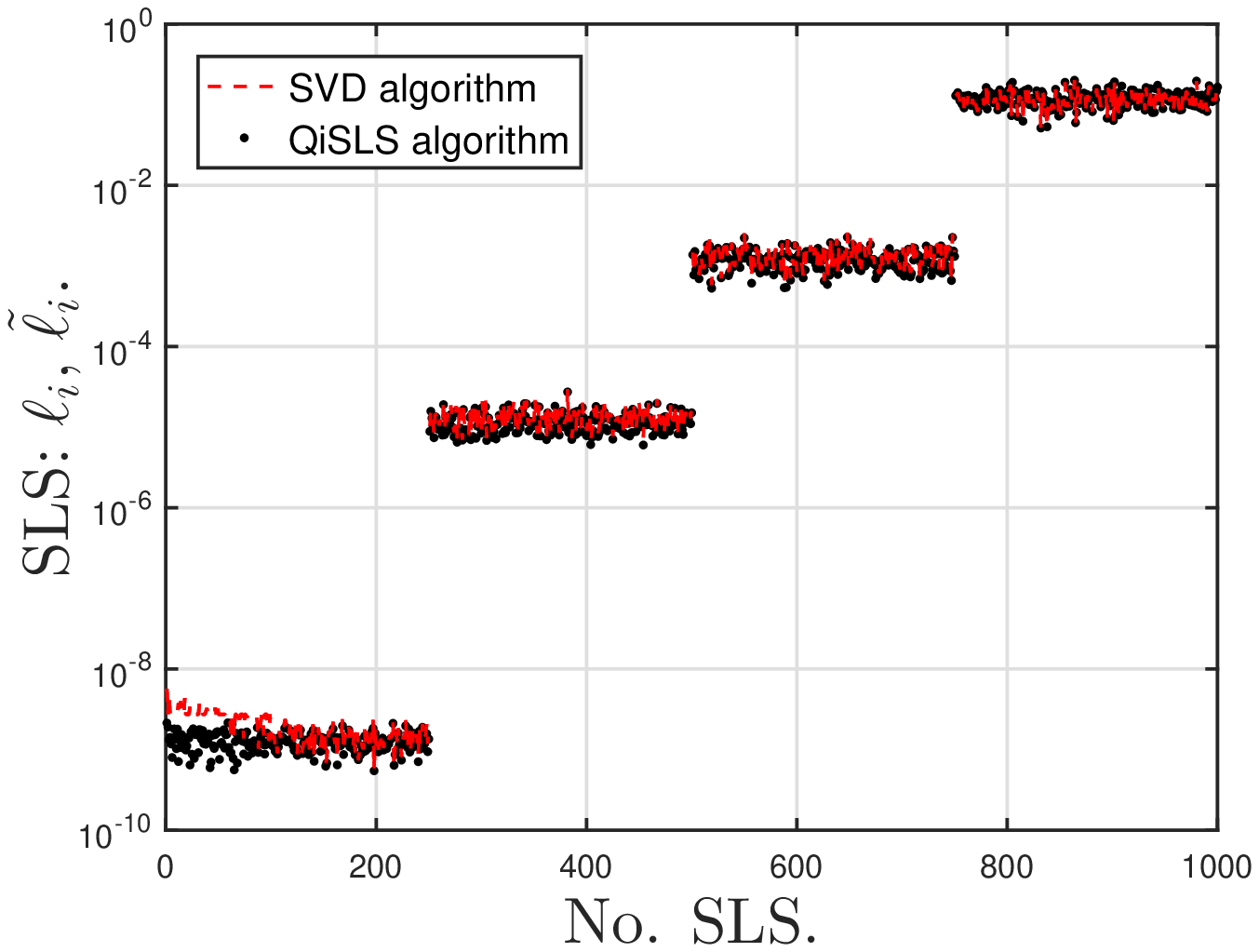}
\caption{SLS of $A$}
\label{QI:SLSMC:fig:sls1:Exp1}
\end{minipage}%
\begin{minipage}[t]{0.5\linewidth}
\centering
\includegraphics[width=2.5in]{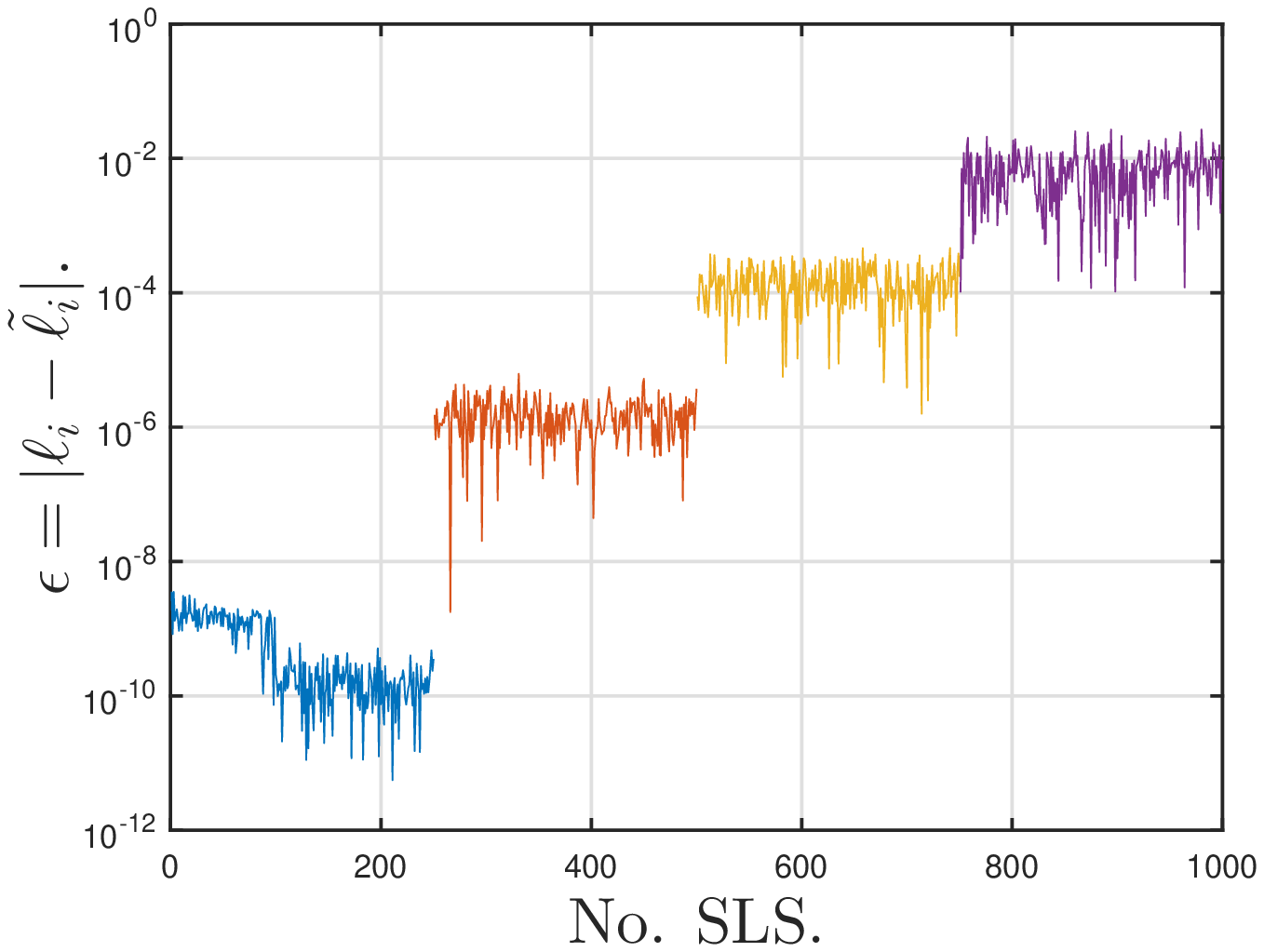}
\caption{Absolute error of SLS $\epsilon$}
\label{QI:SLSMC:fig:err1:Exp1}
\end{minipage}
\end{figure}

\begin{figure}[H]
\begin{minipage}[t]{0.5\linewidth}
\centering
\includegraphics[width=2.5in]{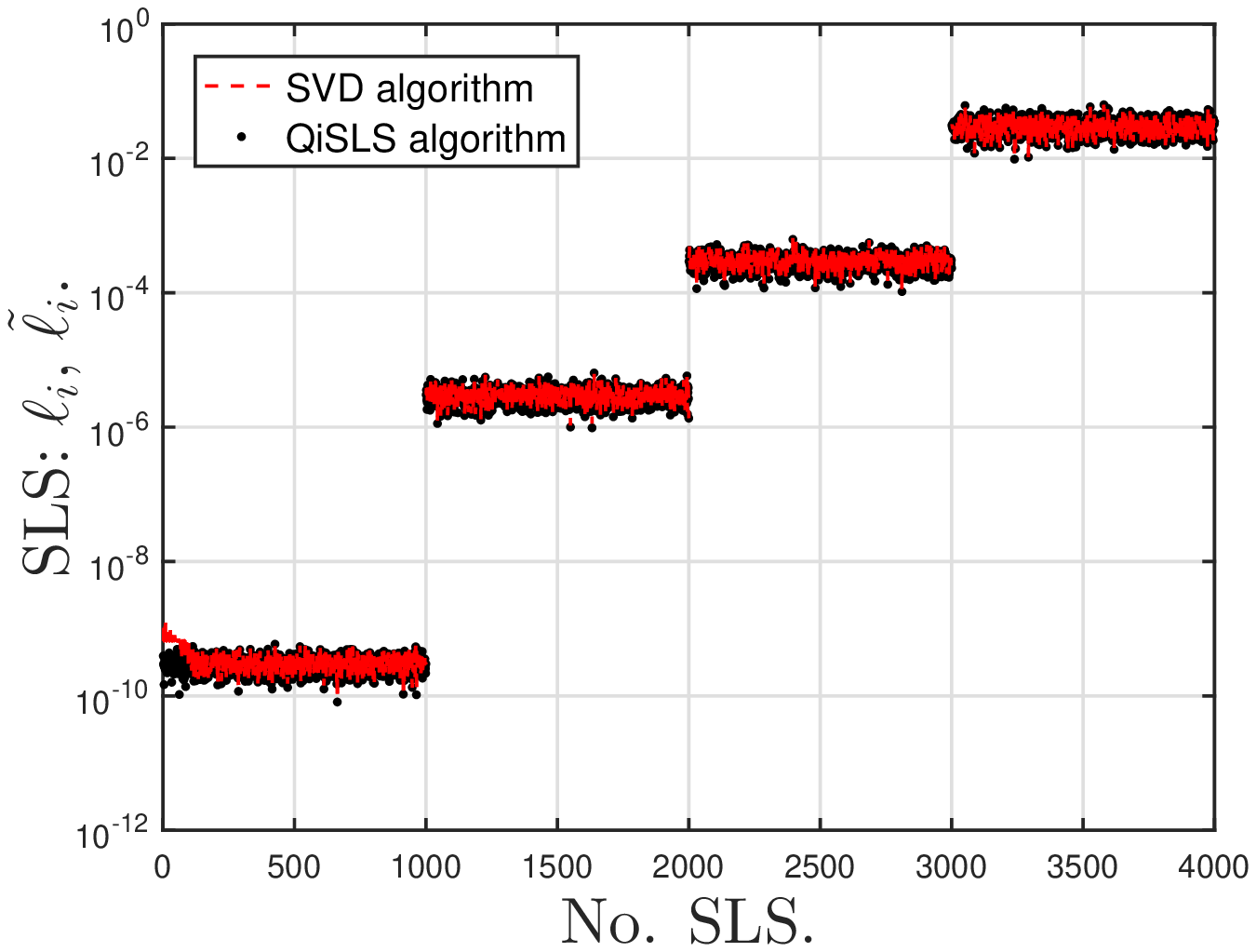}
\caption{SLS of $B$}
\label{QI:SLSMC:fig:sls4:Exp1}
\end{minipage}%
\begin{minipage}[t]{0.5\linewidth}
\centering
\includegraphics[width=2.5in]{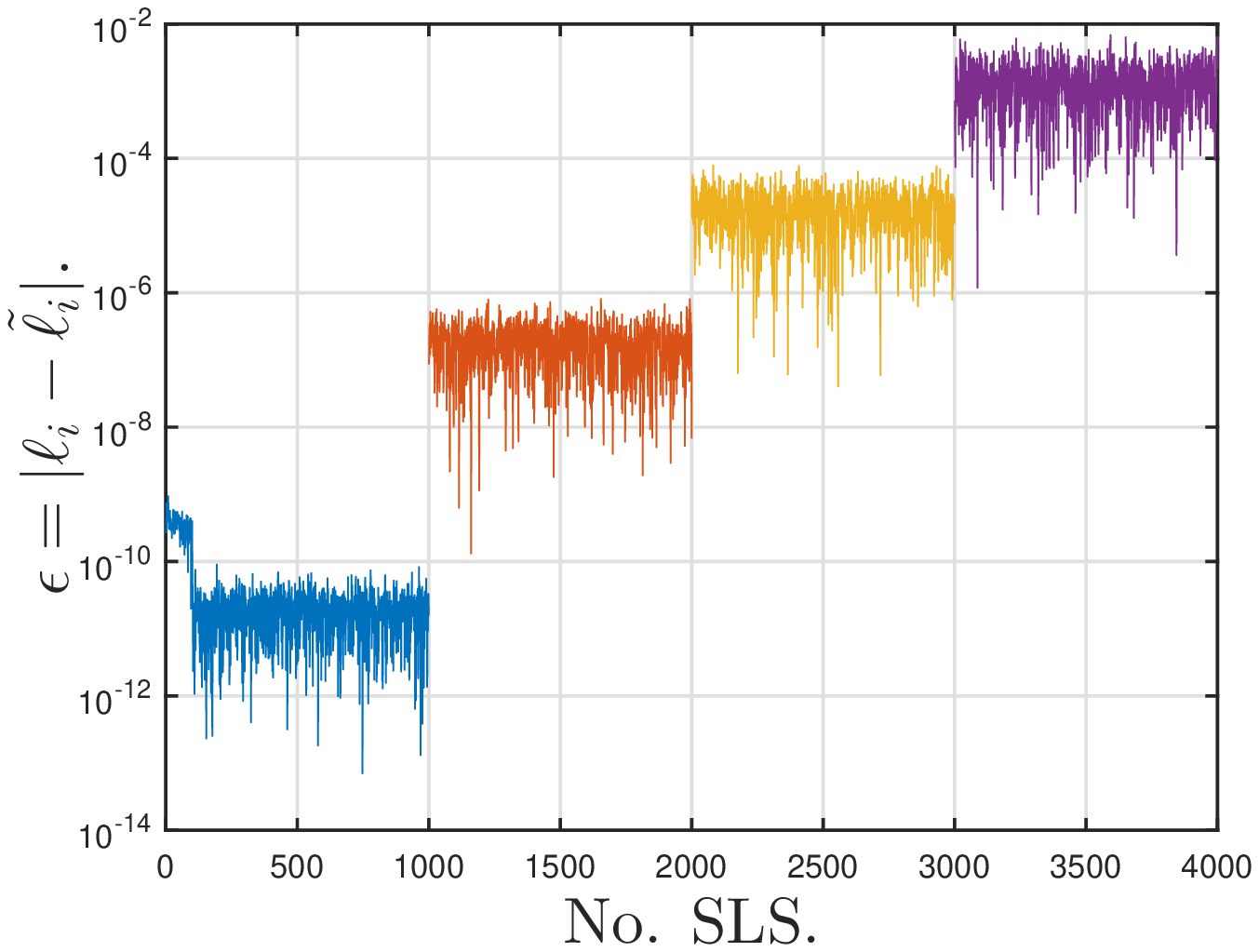}
\caption{Absolute error of SLS $\epsilon$}
\label{QI:SLSMC:fig:err4:Exp1}
\end{minipage}
\end{figure}

For the case Example \ref{QI:SLSMC:Exp1}, the dimensions of $A$ and $B$ are $1000 \times 100$ and $4000 \times 100$, respectively. Setting $n = 70$, we have that the rank of $A$ and $B$ are 30, respectively. Setting the parameters $p = 60$ and $k = 20$, we generate data sets of four different levels of results and make a comparison between our QiSLS algorithm and SVD algorithm. In Figure \ref{QI:SLSMC:fig:sls1:Exp1} and \ref{QI:SLSMC:fig:sls4:Exp1}, the red dashed line denotes the results of statistical leverage scores (SLS) $\ell_{i}$ by SVD algorithm, and the black dot denotes the results of approximate statistical leverage scores $\tilde{\ell}_{i}$ by QiSLS algorithm, respectively. In Figure \ref{QI:SLSMC:fig:err1:Exp1} and \ref{QI:SLSMC:fig:err4:Exp1}, denote the absolute error as $\epsilon = |\ell_{i} - \tilde{\ell}_{i}|$, the four colors $\epsilon$ represent the corresponding data errors. From Figure \ref{QI:SLSMC:fig:sls1:Exp1}-\ref{QI:SLSMC:fig:err4:Exp1}, without loss of generality, for $i \in [m]$, we have that the statistical leverage scores $\tilde{\ell}_{i}$ computed by QiSLS algorithm can make a better approximation on $\ell_{i}$. Similar to previous work  in \cite{DBH21,DHLT20,KMKM21,ADBL20} , we notice that the practical parameter $p$ is much smaller than the theoretical one. Just as \cite{ADBL20}, the performance of quantum inspired algorithm is better than the representations of theoretical complexity bounds.

\begin{Example}\label{QI:SLSMC:Exp2}
In this example, followed the work by in \cite{ADBL20}, we generate a random matrix $A$ of dimension $m \times n$, with rank $r$, and condition number $\kappa$. We sample an $m \times r$ Gaussian random matrix $\Omega$ with entries drawn independently from the standard normal distribution $\mathcal{N}(0,1)$, i.e., $\Omega = randn(m,r)$. The matrix $\Omega$ is generally not orthonormal column, we can make a QR decomposition $\Omega = QR$, where $Q$ is an $m \times r$ orthonormal matrix and $R$ is an $r \times r$ upper triangular matrix. Then, we generate the left singular matrix $U_{A}$ of $A$ by setting $U_{A} = Q$. By a similar method, we can generate the right singular matrix $V_{A}$ of $A$. Finally, for a given condition number $\kappa$, we choose the largest singular value $\sigma_{\rm max}(A)$ uniformly at random in $[a, b]$, and find that the corresponding smallest singular value is $\sigma_{\rm min}(A) = \sigma_{\rm max}(A)/\kappa$. Other singular values are sampled from $(\sigma_{\rm min}(A), \sigma_{\rm max}(A))$. In a word, the matrix $A$ is defined by
\begin{equation}\label{QI:SLSMC:Eq:Exp2}
    \begin{aligned}
    \Omega_1 & = {\rm randn}(m,r), \ [U_{A},R_{u}] = {\rm qr}(\Omega_1), \\
    \Omega_2 & = {\rm randn}(n,r), \ \ [V_{A},R_{v}] = {\rm qr}(\Omega_2), \\
     \sigma_{\rm max}(A)  & = {\rm randperm}(b,a), \ \sigma_{\rm min}(A) = \sigma_{\rm max}(A)/\kappa, \\
    \Sigma_{A} & = {\rm diag[\sigma_{min}(A) + (\sigma_{max}(A) - \sigma_{min}(A))*rand(r,r)]}, \\
    A & = U_{A} \Sigma_{A}V^{T}_{A}.
    \end{aligned}
\end{equation}
\end{Example}

For the case Example \ref{QI:SLSMC:Exp2}, setting $m = 2000$, $n = 500$, $r = 100$, $a = 1$, $b = 1000$ and $\kappa = 1$, then  we generate the low rank matrix $A$ by Eq.\eqref{QI:SLSMC:Eq:Exp2}. For practical applications in not too large datasets, setting $p$ as shown in Algorithm \ref{QI:SLSMC:Subsampling:a1} is rather time consuming. Followed by the work in \cite{DBH21,DHLT20,ADBL20}, setting the parameter $p = 100$, then we make a comparison between SVD algorithm and our QiSLS algorithm with different parameters $k$. The statistical leverage score results of Algorithm \ref{QI:SLSMC:Subsampling:a1} with different parameters $k$ are shown in Figure \ref{QI:SLSMC:fig:sls:Exp2}. In Figure \ref{QI:SLSMC:fig:sls:Exp2}, Similar as the Example \ref{QI:SLSMC:Exp1}, the red dashed line and black dot denote the results of statistical leverage scores $\ell_{i}$ and $\tilde{\ell}_{i}$, respectively. Comparing with the parameters $k = 75$ and $k = 82$, when $k = 88$, except that the two ends (the small and large parts), the remaining parts are well approximated. Without loss of generality, we find that with the increasing properly parameter $k$, the QiSLS algorithm can return a better approximation on SLS $\ell_{i}$.

\begin{figure}[htbp]
\centering
\subfigure{
\begin{minipage}[t]{0.33\linewidth}
\centering
\includegraphics[width=2.2in]{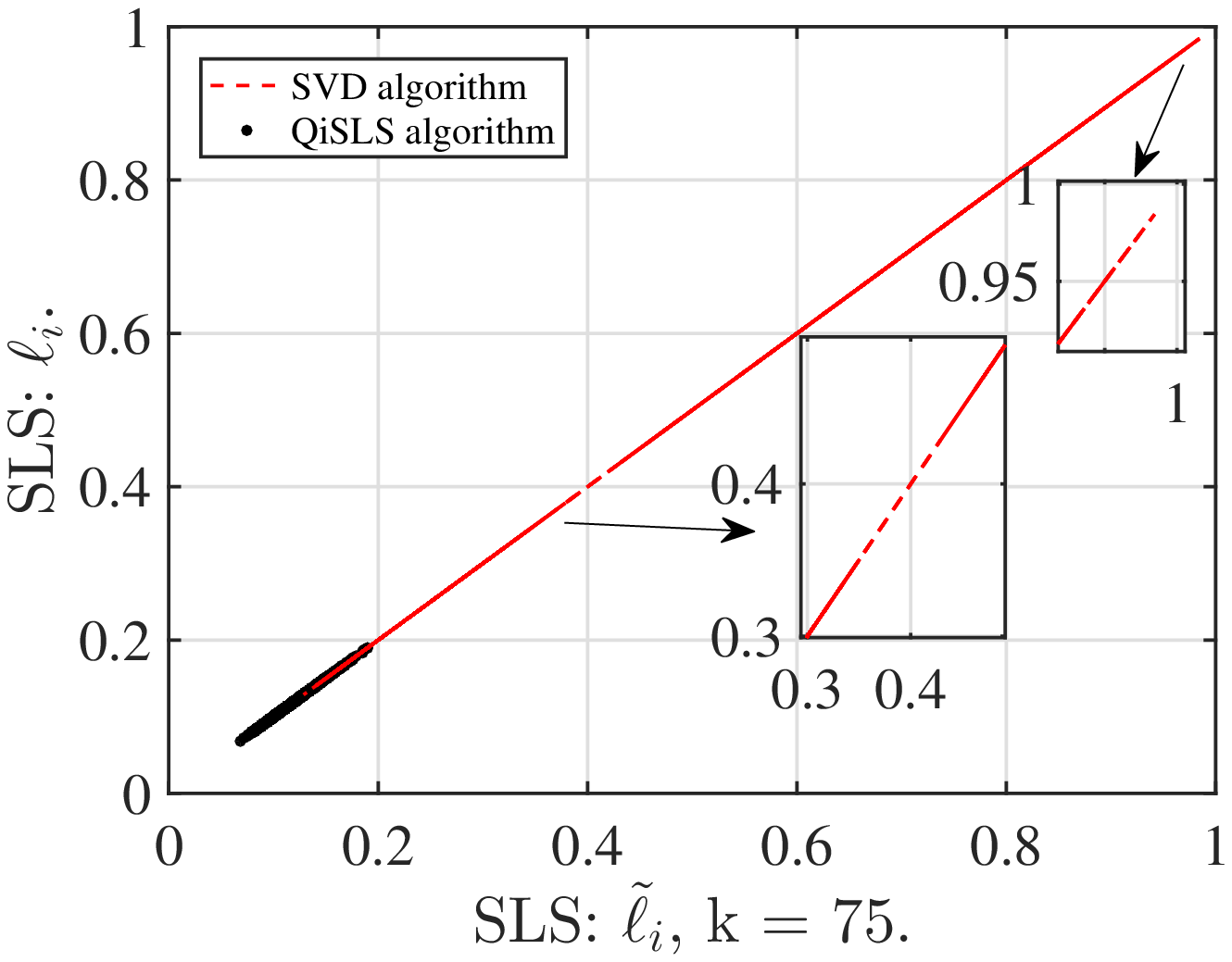}
\end{minipage}%
}%
\subfigure{
\begin{minipage}[t]{0.33\linewidth}
\centering
\includegraphics[width=2.2in]{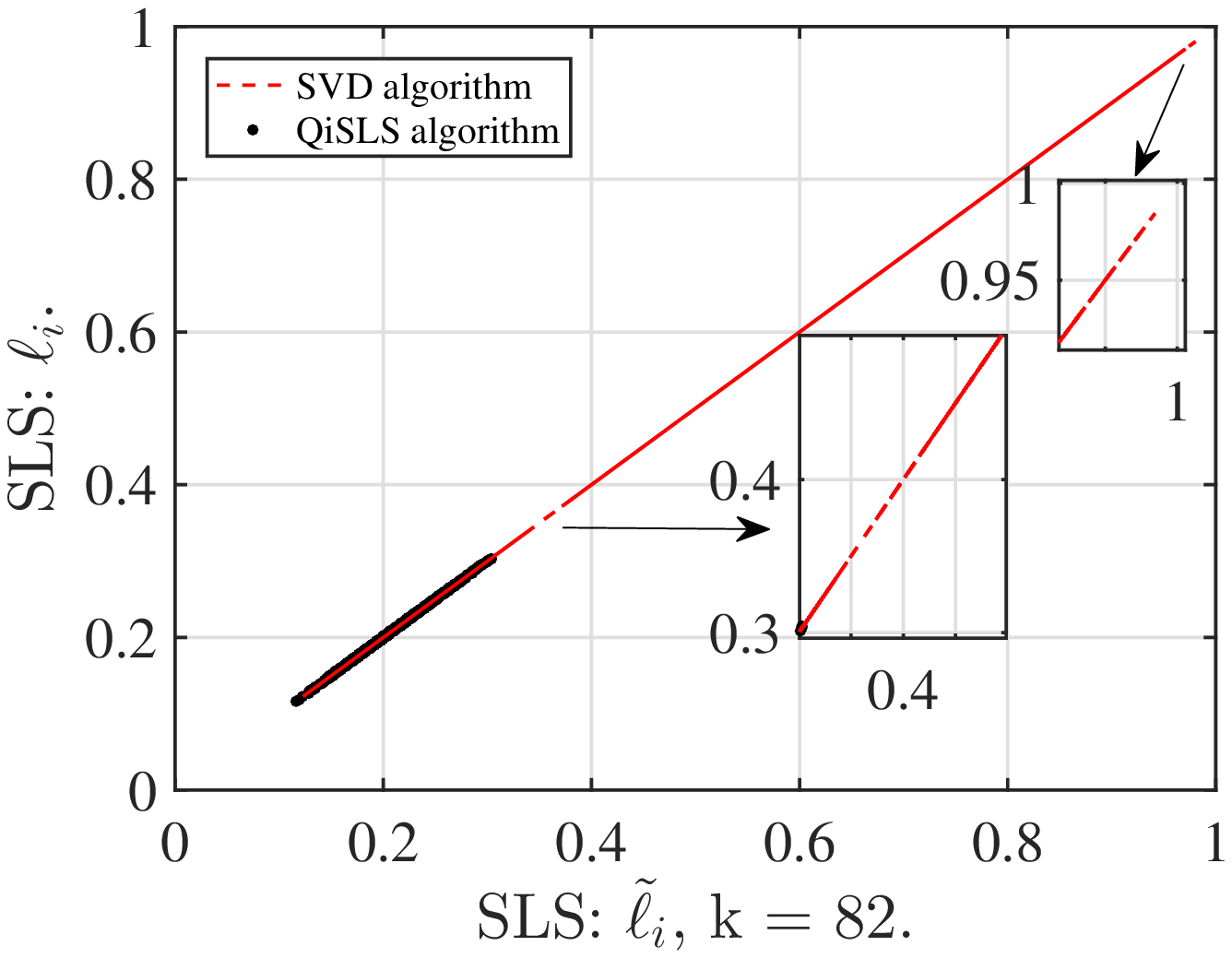}
\end{minipage}%
}%
\subfigure{
\begin{minipage}[t]{0.35\linewidth}
\centering
\includegraphics[width=2.2in]{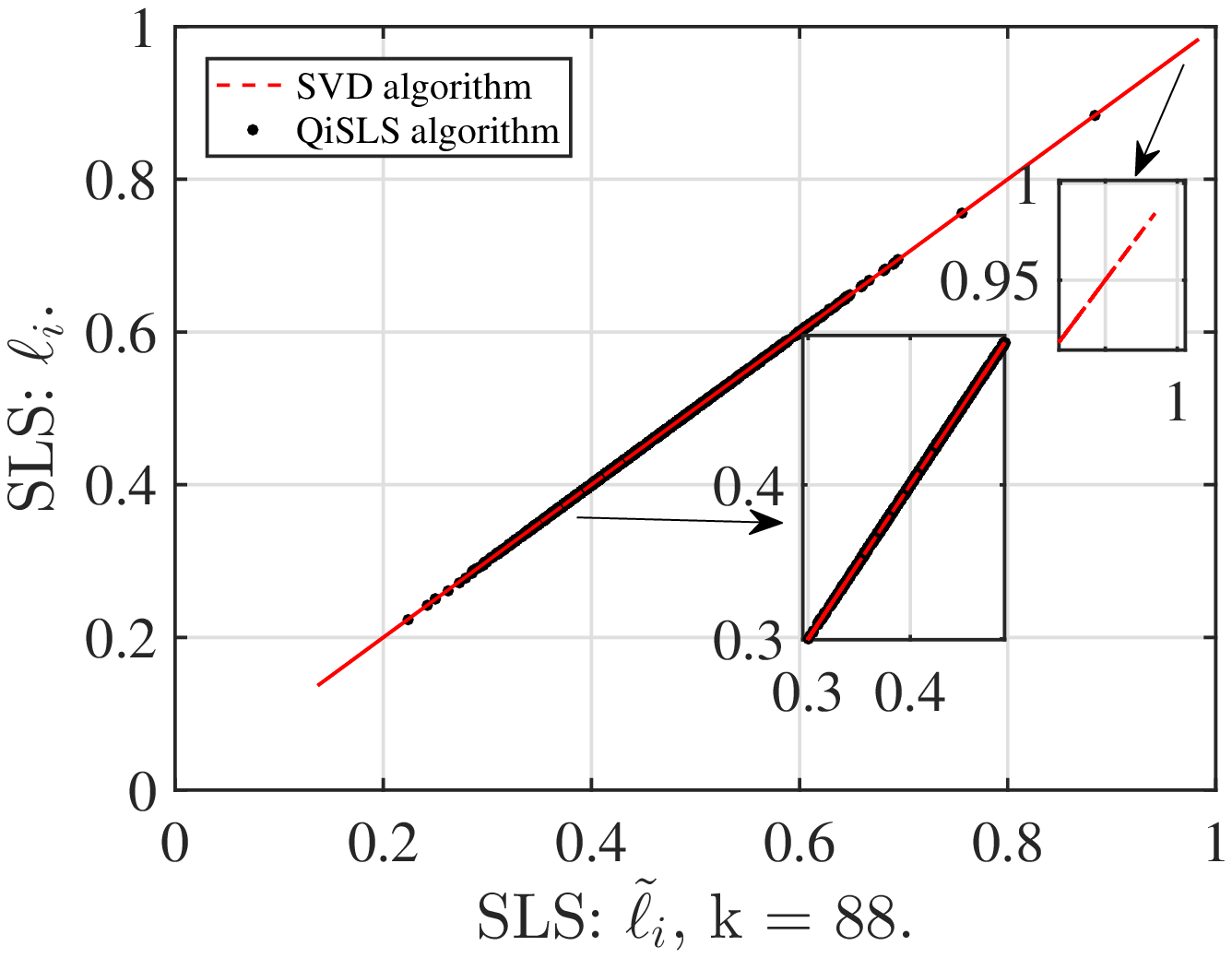}
\end{minipage}
}%
\caption{The statistical leverage scores $\tilde{\ell}_{i}$, $\ell_{i}$ resulted by SVD and QiSLS algorithms with different parameters $k$, respectively.}
\label{QI:SLSMC:fig:sls:Exp2}
\end{figure}

\section{Conclusions}\label{sec:QI:SLSMC:conclusion}
In this paper, based on the sample model and data structure technique, we present a quantum-inspired classical randomized fast approximation algorithm, which dramatically reduces the running time to compute statistical leverage scores of low-rank matrix. It is a generalization of method in \cite{ET18,FKV04}, the core ideas of our implementations of algorithm \ref{QI:SLSMC:Subsampling:a1} is FKV algorithm \cite{FKV04}, but with different parameter setting for choosing the sampling matrix size. It can improve the error of low rank matrix approximation, then we use algorithm \ref{QI:Algorithm:SLSMC:FASLS:a2} to approximate the statistic leverage scores, which combines with inner product estimation method in \cite{ET18}. Next, we theoretically analyze the approximation accuracy and the time complexity of our algorithm. When the matrix size is large, the time complexity of previous work on computing the statistical leverages scores is at least linear in matrix size, our algorithm takes time polynomial in integer $k$, condition number $\kappa$ and logarithm of the matrix size, and we achieve an exponential speedup when the rank of the input matrix is low-rank. Our numerical experiments show that the QiSLS algorithm performs well in practice on large datasets.

Our algorithm still has a pessimistic dependence on condition number $\kappa$ and error parameter $\epsilon$, we will explore other advanced techniques to further improve its dependence on condition number and tighten the error bounds to reduce the computation complexity. Just as the previous quantum-inspired algorithm \cite{ET18,GST18,CLW18,CGLLTW20,JGS19,CLS19,DBH21,DHLT20,ET182}, we do not intend to illustrate the supremacy of quantum computing, we are willing to understand the boundaries between the classical and quantum methods.

\section*{Acknowledgment}
The authors are very thankful to A. Sobczyk for pointing out 4 references.

\end{document}